\documentclass{article}
\usepackage[utf8]{inputenc}

\usepackage{geometry}  	
\usepackage{tikz}
\usepackage{tikz-qtree}
\usepackage{float}

\usepackage{bm}
\usepackage[normalem]{ulem}

\usetikzlibrary{arrows}
\usepackage{mathtools}
\usepackage{graphicx}				
\usepackage[all]{xy}
\usetikzlibrary{arrows, backgrounds, automata, fit, scopes, calc, matrix, positioning, intersections, decorations.pathmorphing, patterns, arrows.meta, decorations.pathreplacing, patterns,shapes,calligraphy, trees}
\usepackage{amsbsy,eucal, amssymb,amsmath}
\usepackage{latexsym,amsfonts,amsthm,dsfont}
\usepackage{authblk,xcolor}
\usepackage{enumerate}

\usepackage{multirow}
\usepackage[colorlinks=true,linkcolor=blue,citecolor=blue,urlcolor=black,bookmarksopen=true]{hyperref}

\newcommand{\set}[1]{\{#1\}}
\newcommand{\pair}[1]{\langle#1\rangle}

\newcommand{\gami}{\text{\bf Gam}$_I$}
\newcommand{\gam}{\text{\bf Gam}}

\newtheorem{definition}{\bf Definition}[section]
\newtheorem{example}[definition]{\bf Example}
\newtheorem{corollary}[definition]{Corollary} 
\newtheorem{proposition}[definition]{Proposition} 
\newtheorem{theorem}[definition]{Theorem} 
\newtheorem{lemma}[definition]{Lemma} 

\title{A categorical representation of games}
\author{Fernando Tohm\'e and Ignacio Viglizzo}
\date{}

\begin{document}
	
	\maketitle
\begin{abstract}
    Strategic games admit a multi-graph representation, in which two kinds of relations, of accessibility and preferences, are used to describe how the players compare the possible outcomes. A category of games with a fixed set of players $\mathbf{Gam}_I$ is built from this representation, and a more general category $\mathbf{Gam}$ is defined with morphisms between representations with different sets of players,  both being {\it complete} and {\it cocomplete}.  The notion of Nash equilibrium can be generalized in this context. We then introduce two subcategories of $\mathbf{Gam}$, \textbf{NE} and $\mathbf{Gam}^{NE}$ in which the morphisms are equilibria-preserving. We illustrate the expressivity and usefulness of this framework with some examples.
\end{abstract}	

	\section{Introduction}
	Game Theory constitutes one of the main foundations of modern analyses of interaction among agents (human and otherwise). Large swaths of Economics and  Evolutionary Biology are game-theoretically characterized. 
    
 Category theory,  a branch of mathematics that provides a high-level, abstract framework for understanding mathematical structures and their relationships was developed in the mid-20th century by Samuel Eilenberg and Saunders Mac Lane \cite{maclane98categories}. It unifies and generalizes concepts across different fields of mathematics, such as algebra, topology, and logic. It studies objects and morphisms (arrows) between them, emphasizing structural relationships rather than specific details of the objects themselves. A modern introduction to the subject may be found in \cite{awodey10category}.
    
   Curiously enough, games have not been studied from a categorical point of view until recently. {A first approach, \cite{vassilakis92economic}, reformulates the classical result of \cite{mertens1985formulation} in terms of the existence of fixed points of functors. Another approach focused on the properties preserved by morphisms among games, can be found in \cite{vorobev1994foundations}}. Lapitsky defines in \cite{Lapitsky1999} game equilibrium notions as functors from a category in which games are the objects to the category of sets. \cite{Jimenez2014} introduces the notion of ``subgame'' (not to be confused with the standard notion of subgame) as the codomain of an endomonomorphism.  \cite{Serafimova2015} present a categorical treatment of how players learn the types of other players.  Unlike the aforementioned contributions, \cite{streufert18category} introduces a category of extensive games, based on their tree structure. More recently, \cite{Vannucci2024} presents a general categoric representation of games as categories of {\em Chu spaces}.
    
    The theory of {\em open games} (\cite{hedges2018morphisms} \cite{ghani2018compositional} \cite{frey2023composing}) pinpoints on one of the main obstacles for a categorical treatment of games. Namely, compositionality is hard to achieve without retrofitting the notion of game with an approach based on {\em lenses} and other concepts of the recently developed field of categorical {\em optics} \cite{hedges2017coherence} \cite{capucci2021towards}.

	Our goal here differs from {this prevailing approach to games in the literature on applied category theory}. Instead of only focusing on the composition of games as a way of building up new ones, we abstract away payoffs and actions representing them through binary relations on the sets of possible outcomes of the game, {a standpoint close to the {\em order graphs} introduced in \cite{goforth2004topology}}. Then each game becomes a multigraph with relations indexed by the set of players.  The vertices correspond to the outcomes, and for each player, we consider two sorts of edges, one representing the preferences, and another indicating which outcomes are effectively reachable through the actions (accessibility relation).  While actions are not explicitly included in the model, this presentation provides a more flexible and expressive framework for some aspects of game theory.

 	We think that our representation of strategic games is at the same time close enough to the one commonly used by game theory practitioners, yet flexible enough to allow us to build the basic category-theoretic constructions. The key for the categorical constructions in this paper are the properties of the accessibility and preference relations. These relations suffice to define Nash equilibria and thus facilitate the characterization of a category in which morphisms among their objects preserve the equilibria.

	In the next section, we introduce the representation of strategic games we will use in the paper. In section~\ref{cat} we introduce two categories in which these representations of games are the objects, one with a fixed set of players, \gami, and another including games with different sets of players, \gam. We prove that both these categories are complete, cocomplete, and have exponentials. In section \ref{equilibria} we define subgames and generalize the notion of Nash equilibrium arising from the interaction between the preference and accessibility relations. 
 
 Finally, in section \ref{sectNash} we investigate two subcategories of \gam\  in which the morphisms preserve Nash equilibria.

	\section{Multi-Graph Representation of Strategic Games}
	
	We are going to consider a category whose objects model some aspects of  {\em strategic games} (\cite{osborne94course}):
	\begin{definition}\label{definitionStartegicGame}
		A  strategic game is a structure of the form $G = \langle I, \{A_i\}_{i \in I}, \{\pi_i\}_{i \in I} \rangle$  where $I=\set{1,\ldots, n}$ is a set of {\em players} and $A_i, i\in I$ is a finite set of {\em strategies} for each player, and  $\pi_i : \prod_{i \in I} A_i \rightarrow \mathbb{R} $ is player $i$'s payoff.
	\end{definition}

 A critical component in defining a category is specifying the class of {\em morphisms} among its objects. For the category of games, this involves developing a representation of {\em games} that addresses certain limitations of more direct approaches. Specifically, a direct approach led to overly rigid categories, where the only morphisms were inclusions and isomorphisms. To overcome this, our representation replaces each profile $a = (a_1, \ldots, a_n) \in \prod_{i \in I} A_i$ with the outcome $o$ that results when all players play their actions in $a$. This representation also models two essential features of games: the ability of {\em individual} agents to influence outcomes and their preferences over them.

This alternative representation of a game can be seen as defining a {\em multigraph}. That is, as a set of nodes with multiple edges between them (including loops at the nodes \cite{Tutte98gt}). Each node corresponds to an outcome of the game and between any pair of nodes there can exist, for each player, two types of edges. One type is undirected, corresponding to the possibility that the player to which it belongs may change from one outcome to the other. The second type is directed and represents that one outcome is preferred by the player over the other.

More formally, and similarly to what is observed in \cite{schipper02undergraduates}, we have:
	
	\begin{definition}\label{definitionGame}
		A {\em game} is $G = \langle I, O, \{R_i\}_{i \in I}, \{\preceq_i\}_{i \in I} \rangle$, where $I$ is a set of {\em players} and $O$ is a  set of {\em outcomes}.
		For each $i \in I$, $R_i$ is a reflexive and symmetric binary {\em accessibility} relation, a subset of $O \times O$. In turn, $\preceq _i\subseteq O\times O$ is a preorder (this is, a reflexive and transitive relation) representing player  $i$'s {\em preferences} over outcomes.
	\end{definition}

	To see how games defined in this way have some of the same properties as those in the more standard approach in Game Theory above, take as set $O$ of outcomes the list of the combined actions of all the players, that is, the {\em profiles} of actions $(a_1, \ldots, a_i, \ldots, a_n)$ in  $\prod_{i \in I} A_i=O$.  Then, given any outcome $o = (a_1, \ldots, a_i, \ldots, a_n)$, a unilateral change of the action chosen by $i$ (while all other players stay put), say from $a_i$ to $a'_i$, yields $o' = (a_1, \ldots, a'_i, \ldots, a_n)$. We can understand this as indicating that the pair $\langle o, o'\rangle$ belongs to a relation of accessibility $R_i$, indicating how unilateral choices of a player $i$ can lead from an outcome to other ones. In this way, $R_i$ captures the result of exerting the actions of $A_i$. 
	
	We assume that the relations $R_i$ are reflexive, i.e. $\langle o, o\rangle \in R_i$ for every $i\in I$ and $o\in O$. This can be justified by considering the case of strategic games: if $o = (\ldots, a_i, \ldots)$ and $i$ does not change her choice $a_i$, the result is again $o$. Analogously, we assume that $R_i$ is symmetric. In the context of strategic games, this is justified by the fact that from $o = (a_1, \ldots, a_i, \ldots, a_n)$, a unilateral change from $a_i$ to $a'_i$ yields $o' = (a_1, \ldots, a'_i, \ldots, a_n)$ just as from $o'$, a unilateral change from $a'_i$ to $a_i$ yields $o$. That is, if $\langle o, o' \rangle \in R_i$ then $\langle o', o \rangle \in R_i$. We will also use freely the infix notation for the relations $R_i$, writing $oR_io'$ instead of $\pair{o,o'}\in R_i$.
    
    We can illustrate the characterization of accessibility relations in a strategic game using the following example.
	
	\begin{example}~\label{PD}
		Let us consider the well-known Prisoner's Dilemma  (PD) as a strategic game given by the matrix:
		
		\begin{center}
			\begin{tabular}{l|c|c}
				&   C   &   D   \\ \hline
				C & (-1,-1) & (-3,0) \\ \hline
				D & (0,-3) & (-2,-2)
			\end{tabular}
		\end{center}	
		
		In this game ${PD}$, $I=\set{1,2}$, where 1 chooses rows and 2 columns.  Here $O_{PD}=\set{(C,C),(C,D),$ $(D, C),(D,D)}$, where $C$ stands for `Cooperate' and $D$ stands for `Defect'    .  Figure 1 presents the corresponding multi-graph representation of this game. We can see, for instance, that player 1 can change  unilaterally from $D$ to $C$, so $(D,C)$ is accessible from $(C,C)$ and vice-versa,  and the same for  $(D,D)$ and  $(C,D)$. The preferences are such that for instance $(C,C) \preceq_1 (C,D)$ since $\pi_1(C,C) = -1$ and $\pi_1(C,D) =0$.

	\begin{figure}[H]
	\centering
	
	\begin{tikzpicture}[scale=1,
	DashedR/.style={
		dashed,
		->,>=stealth',
		shorten >=1pt,
		auto,
		thick,
		draw=red!60},
	DashedB/.style={
		dashed,
		->,>=stealth',
		shorten >=1pt,
		auto,
 		thick,
		draw=blue!60},
	FillR/.style={
		thick,
		draw=red!60},
	FillB/.style={
		thick,
		draw=blue!60}]
	
	\tikzstyle{every state}=[fill=gray!30!white,draw=none]
	
	\coordinate (A) at (0,0);
	\coordinate (B) at (4,0);
	\coordinate (C) at (4,4);
	\coordinate (D) at (0,4);
	
	\node[state] (N1) at (A) {DC};
	\node[state] (N2) at (B) {DD};
	\node[state] (N3) at (C) {CD};
	\node[state] (N4) at (D) {CC};
	
	\path (N1) edge [FillR] (N2)
	(N1) edge [DashedR, bend left = 15] (N2)
	(N2) edge [DashedB, bend left = 15] (N1)
	(N2) edge [FillB] (N3)
	(N2) edge [DashedR, bend left = 15] (N3)
	(N3) edge [DashedB, bend left = 15] (N2)
	(N1) edge [FillB] (N4)
	(N1) edge [DashedR, bend left = 15] (N4)
	(N4) edge [DashedB, bend left = 15] (N1)
	(N4) edge [FillR] (N3)
	(N4) edge [DashedR, bend left = 15] (N3)
	(N3) edge [DashedB, bend left = 15] (N4)
	
	(N1) edge [DashedR, bend right = 8] (N3)
	(N3) edge [DashedB, bend right = 8] (N1)
	
	(N2) edge [DashedB, bend left = 8] (N4)
	(N2) edge [DashedR, bend right = 8] (N4);
	
	\draw[DashedB](N1) to [out=195,in=165,looseness=8] (N1);
	\draw[FillB](N1) to [out=240,in=210,looseness=8] (N1);
	\draw[FillR](N1) to [out=285,in=255,looseness=8] (N1);
	\draw[DashedR](N1) to [out=330,in=300,looseness=8] (N1);
	
	\draw[DashedB](N2) to [out=240,in=210,looseness=8] (N2);
	\draw[FillR](N2) to [out=285,in=255,looseness=8] (N2);
	\draw[DashedR](N2) to [out=330,in=300,looseness=8] (N2);
	\draw[FillB](N2) to [out=375,in=345,looseness=8] (N2);
	
	\draw[DashedR](N3) to [out=60,in=30,looseness=8] (N3);
	\draw[FillR](N3) to [out=105,in=75,looseness=8] (N3);
	\draw[DashedB](N3) to [out=150,in=120,looseness=8] (N3);
	\draw[FillB](N3) to [out=375,in=345,looseness=8] (N3);
	
	\draw[FillB](N4) to [out=60,in=30,looseness=8] (N4);
	\draw[FillR](N4) to [out=105,in=75,looseness=8] (N4);
	\draw[DashedB](N4) to [out=150,in=120,looseness=8] (N4);
	\draw[DashedR] (N4) to [out=195,in=165,looseness=8] (N4);
	
	\end{tikzpicture}
	
	\caption{Multi-graph representing the Prisoner's Dilemma in Example~~\ref{PD}.
	Blue and red lines correspond to players 1 and 2, respectively. Full (undirected) lines correspond to accessibility relations while dashed ones (directed) represent preferences.}
	\label{figPD}
\end{figure}
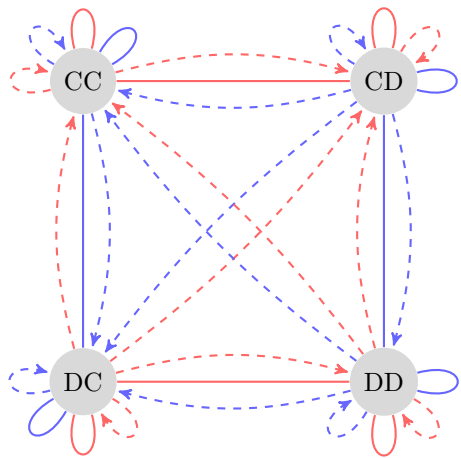

			\end{example}

	Notice that in this standard interpretation of the game, each $R_i$ is transitive. Lifting this requirement allows us to model other situations:
	
	\begin{example}~\label{off}
		Consider the following strategic game, $G_1$:
		
		\begin{center}
			\begin{tabular}{l|c|c}
				&   L   &   R   \\ \hline
				T & (0,0) & (-1,2) \\ \hline
				D & (2,-1) & (0,0)
			\end{tabular}
		\end{center}	
		
		\noindent where again player 1 chooses rows and 2 columns. In this game, the outcomes $(T,L)$ and $(D,R)$ both lead to the same payoff of $0$ for both players. Figure~\ref{figg1} depicts the corresponding multi-graph representation.

         \begin{figure}[H]
	\centering
	
	\begin{tikzpicture}[scale=1,
	DashedR/.style={
		dashed,
		->,>=stealth',
		shorten >=1pt,
		auto,
		thick,
		draw=red!60},
	DashedB/.style={
		dashed,
		->,>=stealth',
		shorten >=1pt,
		auto,
 		thick,
		draw=blue!60},
	FillR/.style={
		thick,
		draw=red!60},
	FillB/.style={
		thick,
		draw=blue!60}]
	
	\tikzstyle{every state}=[fill=gray!30!white,draw=none]
	
	\coordinate (A) at (0,0);
	\coordinate (B) at (4,0);
	\coordinate (C) at (4,4);
	\coordinate (D) at (0,4);
	
	\node[state] (N1) at (A) {DL};
	\node[state] (N2) at (B) {DR};
	\node[state] (N3) at (C) {TR};
	\node[state] (N4) at (D) {TL};
	
	\path (N1) edge [FillR] (N2)
	(N1) edge [DashedR, bend left = 15] (N2)
	(N2) edge [DashedB, bend left = 15] (N1)
	(N2) edge [FillB] (N3)
	(N2) edge [DashedR, bend left = 15] (N3)
	(N2) edge [DashedB, bend right = 15] (N3)
	(N1) edge [FillB] (N4)
	(N1) edge [DashedR, bend left = 15] (N4)
	(N4) edge [DashedB, bend left = 15] (N1)
	(N4) edge [FillR] (N3)
	(N4) edge [DashedR, bend left = 15] (N3)
	(N3) edge [DashedB, bend left = 15] (N4)
	
	(N1) edge [DashedR, bend right = 8] (N3)
	(N3) edge [DashedB, bend right = 8] (N1)
	
	(N2) edge [DashedB, bend left = 8] (N4)
    (N4) edge [DashedB, bend right = 8] (N2)
	(N2) edge [DashedR, bend right = 8] (N4)
    (N4) edge [DashedR, bend left = 8] (N2);
	
	\draw[DashedB](N1) to [out=195,in=165,looseness=8] (N1);
	\draw[FillB](N1) to [out=240,in=210,looseness=8] (N1);
	\draw[FillR](N1) to [out=285,in=255,looseness=8] (N1);
	\draw[DashedR](N1) to [out=330,in=300,looseness=8] (N1);
	
	\draw[DashedB](N2) to [out=240,in=210,looseness=8] (N2);
	\draw[FillR](N2) to [out=285,in=255,looseness=8] (N2);
	\draw[DashedR](N2) to [out=330,in=300,looseness=8] (N2);
	\draw[FillB](N2) to [out=375,in=345,looseness=8] (N2);
	
	\draw[DashedR](N3) to [out=60,in=30,looseness=8] (N3);
	\draw[FillR](N3) to [out=105,in=75,looseness=8] (N3);
	\draw[DashedB](N3) to [out=150,in=120,looseness=8] (N3);
	\draw[FillB](N3) to [out=375,in=345,looseness=8] (N3);
	
	\draw[FillB](N4) to [out=60,in=30,looseness=8] (N4);
	\draw[FillR](N4) to [out=105,in=75,looseness=8] (N4);
	\draw[DashedB](N4) to [out=150,in=120,looseness=8] (N4);
	\draw[DashedR] (N4) to [out=195,in=165,looseness=8] (N4);
	
	\end{tikzpicture}
	
	\caption{Multi-graph representation of  $G_1$ from Example~~\ref{off}.
	Blue and red lines correspond to players 1 and 2, respectively. Full (undirected) lines correspond to accessibility relations while dashed ones (directed) represent preferences.}
	\label{figg1}
\end{figure}

       If, instead, the players identify $(T,L)$ and $(D,R)$ we could represent this result as a single outcome $o$. Notice that such identification amounts to forgetting how each outcome may have been reached, and thus the accessibility relations are conflated. Figure~\ref{figNT2} represents this game.
		
		\begin{figure}[H]
			\centering
			
			\begin{tikzpicture}[scale=1,
			DashedR/.style={
				dashed,
				->,>=stealth',
				shorten >=1pt,
				auto,
				thick,
				draw=red!60},
			DashedB/.style={
				dashed,
				->,>=stealth',
				shorten >=1pt,
				auto,
				thick,
				draw=blue!60},
			FillR/.style={
				thick,
				draw=red!60},
			FillB/.style={
				thick,
				draw=blue!60}]
			
			\tikzstyle{every state}=[fill=gray!30!white,draw=none]
			
			\coordinate (A) at (0,0);
			\coordinate (C) at (4,4);
			\coordinate (D) at (0,4);
			
			\node[state] (N1) at (A) {DL};
			\node[state] (N3) at (C) {TR};
			\node[state] (N4) at (D) {o};
			
			\path
			(N1) edge [FillB, bend left = 5] (N4)
			(N4) edge [FillR, bend left = 5] (N1)
			(N1) edge [DashedR, bend left = 15] (N4)
			(N4) edge [DashedB, bend left = 15] (N1)
			
			(N3) edge [FillB, bend left = 5] (N4)
			(N4) edge [FillR, bend left = 5] (N3)
			(N4) edge [DashedR, bend left = 15] (N3)
			(N3) edge [DashedB, bend left = 15] (N4)
			
			(N1) edge [DashedR, bend right = 8] (N3)
			(N3) edge [DashedB, bend right = 8] (N1);
			
			\draw[DashedB](N1) to [out=195,in=165,looseness=8] (N1);
			\draw[FillB](N1) to [out=240,in=210,looseness=8] (N1);
			\draw[FillR](N1) to [out=285,in=255,looseness=8] (N1);
			\draw[DashedR](N1) to [out=330,in=300,looseness=8] (N1);
			
			\draw[DashedR](N3) to [out=60,in=30,looseness=8] (N3);
			\draw[FillR](N3) to [out=105,in=75,looseness=8] (N3);
			\draw[DashedB](N3) to [out=150,in=120,looseness=8] (N3);
			\draw[FillB](N3) to [out=375,in=345,looseness=8] (N3);
			
			\draw[FillB](N4) to [out=60,in=30,looseness=8] (N4);
			\draw[FillR](N4) to [out=105,in=75,looseness=8] (N4);
			\draw[DashedB](N4) to [out=150,in=120,looseness=8] (N4);
			\draw[DashedR] (N4) to [out=195,in=165,looseness=8] (N4);
			
			\end{tikzpicture}
			
			\caption{Multi-graph representing the game $G_1$ with set of outcomes $O_1=\set{o,DL,TR}$ from Example~~\ref{off}, in which the accessibility relation is not transitive. Blue and red lines correspond to players 1 and 2, respectively. Full lines correspond to accessibility relations and dashed ones to preferences.}
			\label{figNT2}
		\end{figure}
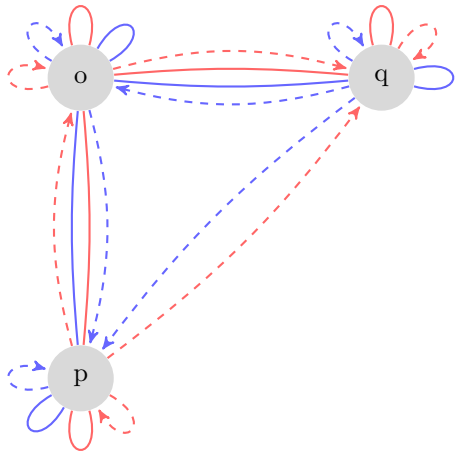
	\end{example}

As said, the representation of games as multigraphs is intended to facilitate the specification of a category of games. This is based on the properties of the two types of edges or relations among outcomes.

 For each player $i$, the pair $\pair{O,R_i}$ can be regarded as an object in the category  $\bf EndoRel$, which is the category of sets endowed with a binary relation on them (an \textit{endorelation}), with morphisms the set functions that preserve the relation. That is, $f:\pair{A,R}\to \pair{B,S}$ is a morphism in $\bf EndoRel$ if for every $a,a'\in A$, if $\pair{a,a'}\in R$ then $\pair{f(a),f(a')}\in S$.  We have chosen for our representation to be a bit more specific, asking that the binary relations are reflexive and symmetric. This can also be interpreted as a category of \textit{undirected graphs}.

On the other  hand, the pairs $\pair{O,\preceq_i}$ can be seen as objects in  $\bf PreOrd$, the category of  preordered sets with  monotone functions as morphisms.

 This new representation, while an alternative to strategic games, effectively captures diverse kinds of games, as shown in the following examples:

    \begin{example}
    {Any extensive-form game has the form $G = \langle I, \mathcal{H}, P, A, \{\mathcal{I}_i\}_{i \in I}, \{\pi_i\}_{i \in I}\rangle$, where $I$ is the set of players, $\mathcal{H}$ the set of {\em histories} in the game, $P$ a function that to each non-terminal history $h \in \mathcal{H}$ assigns a player $P(h) \in I$ and a set of possible actions that this player can choose, $A(h)$. Each $\mathcal{I}_i$ is a partition of $\{h \in \mathcal{H} : P(h)=i\}$. Finally, $\pi_i$ is a real-valued function defined over {\em terminal} histories in $\mathcal{H}$ \cite{osborne94course}. Each extensive-form game $G$ can be translated into a strategic-form game $G^{'}$ in which each $a_i \in A_i$ is a function that assigns to each part $I_i \in \mathcal{I}_i$ an action that can be exerted at the end of a history $h \in A_i$ (such $a_i$ is called a {\em strategy}). The payoff function is $\hat{\pi}_i$, which assigns to each $a \in \prod_{i \in I} a_i$ (which corresponds to a single terminal history $h$) the value $\pi_i(h)$. $G^{'}$ can be naturally translated into a corresponding multi-graph version according to Definition~\ref{definitionGame}.}
    \end{example}
   
   \begin{example}
       
  Given any strategic game $G = \langle I, \{A_i\}_{i \in I}, \{\pi_i\}_{i \in I} \rangle$ can be extended to a {\em mixed-extension} game $G^e =  \langle I, \{\Delta A_i\}_{i \in I}, \{E\pi_i\}_{i \in I} \rangle$, where each $\sigma_i \in \Delta A_i$ is $\sigma_i: A_i \rightarrow [0,1]$ such that $\sum_{a \in A_i} \sigma_i(a) = 1$, i.e. a probability distribution over $A_i$ with countable support. The payoff function is an {\em expected utility} $E\pi_i(\sigma_1, \ldots, \sigma_{|I|}) = \sum_{a \in \prod_{i \in I} A_i} \sigma_1(a_1) \ldots \sigma_{|I|}(a_{|I|}) \pi_i (a_1, \ldots, a_{|I|}) $. In the multi-graph corresponding to $G^e$ the set of outcomes is
$$\Delta O = \{\sigma \in \prod_{i \in I}\Delta A_i\}$$
 where for each $\ i \in I$, the accessibility relation $R^{\Delta O}_i$, can be defined as follows. Given $\sigma, \tau \in \Delta O$,
$$  \sigma R^{\Delta O}_i \tau \ \mbox{iff} \ \sigma_{j} = \tau_{j} \text{ for every } j\neq i.$$
That is, $\sigma$ and $\tau$ can be accessed from each other by $i$ if they can differ only in $i$'s mixed strategies over $A_i$. It follows immediately that $R^{\Delta O}_i$ is reflexive and symmetric.
With respect to the preferences, for each $i \in I$, given $\sigma, \tau \in \Delta O$,
$$ \sigma \preceq_i \tau 
 \ \mbox{iff} \ E\pi_i(\sigma) \leq  E\pi_i(\tau)$$
It is immediate that for each $i \in I$, $\preceq_i$ satisfies reflexivity and transitivity.
 \end{example}

   \begin{example}
   Each incomplete information game $G^{\mu}$ adds to definition~\ref{definitionStartegicGame} two elements, namely a set $\Omega$ of {\em states} (profiles of {\em types} of the players) and a map $\mu$ that assigns to each  state $\omega \in \Omega$ and each player $i \in I$ a set of actions $A^{\omega}_i \subset A_i$ and a payoff function $\pi^{\omega}_i$ for $i$ in state $\omega$. This assignment has an associated probability $\hat{\mu}(\omega, i)$. $G^{\mu}$ can be reframed as a {\em mixed-extension} game\cite{harsanyi67games1} (see also \cite{pivato2024categorical}).
  \end{example}
  
In general, a larger class of interactive decision-making problems can represented by multi-graphs, since preferences, being morphisms in {\bf PreOrd}, allow for incomparabilities among outcomes, unlike the usual representation based on payoff functions, which impose a linear order on outcomes.

In the next section we define different categories in which the objects are multi-graph versions of strategic games and the morphisms preserve the accessibility and the preference relations, acting as morphisms in {\bf EndoRel} and {\bf PreOrd}, respectively.
		
	\section{Categories of Multi-Graph Representation of Games}~\label{SectGamI} \label{cat}

\subsection{Games with a fixed set of players}

As a first approach to the categorification of the class of multi-graph representations of strategic games presented in Definition~\ref{definitionGame}, we consider the subclass of them in which the set of players is the same set $I$ and define the category $\mathbf{Gam}_I$:

\begin{definition}\label{definitionGamI}
	Each object in the category \gami \ is  $G = \langle O, \{R_i\}_{i \in I}, \{\preceq _i\}_{i \in I} \rangle$ as in Definition \ref{definitionGame}, and if $G' = \langle O', \{R'_i\}_{i \in I}, \{\preceq' _i\}_{i \in I} \rangle$ is also an object, a morphism  $f:G\to G'$  is a function $f:O\to O'$ such that for all $i\in I$, $f$ preserves $R_i$ and $\preceq_i$, that is, for all $o,p\in O$:
	\[oR_ip \text{ implies }  f(o)R_i'f(p), \] 
	and
	\[o\preceq_ip\text{ implies } f(o)\preceq'_if(p).\]
\end{definition}

That is, $f$ is a morphism in \gami \  if and only if for each $i\in I$, $f:\pair{O,R_i}\to\pair{O',R'_{i}}$ and $f:\pair{O,\preceq_i}\to\pair{O',\preceq'_{i}}$ are {\bf EndoRel} and {\bf PreOrd} morphisms, respectively. 

In more intuitive terms, a morphism from $G$ to $G'$ maps the outcomes of $G$  on those of $G'$, preserving the preferences of all the players, such that for every alternative option for any player in $G$  is mapped to a viable one for the player in $G'$.

\begin{example} \label{example morphisms}The games ${PD}$ from Example~\ref{PD} and $G_1$ of Example \ref{off} (in the version depicted in Figure~\ref{figNT2}) are objects in the category $\mathbf{Gam}_{\{1,2\}}$.  
A morphism $g:PD\to G_1$ is given by $g(C,C) = o = g(D,D)$, $g(D,C) = (D,L)$ and $g(C,D) = (T,R)$. We can easily check that $g$ preserves $R_1, R_2, \preceq_1$, and $\preceq_2$.
\end{example}

It is clear that the composition of morphisms is a  morphism, and that the identity functions over the sets of outcomes are morphisms. Also, the composition is associative, since it is just the composition of set functions preserving relations. Thus we have a category \gami \  of games with a fixed set of players $I$.

One of the advantages of the categorical approach is that we can identify (up to isomorphism) the objects with desirable universal properties in \gami. These objects can be obtained as in the categories $\bf EndoRel$  and $\bf PreOrd$. For those categories, the proofs are worked out in  \cite{viglizzo23basic},
and can be easily adapted to the case we deal with here of symmetric and reflexive relations.

 As a first example we have: 
\begin{itemize}
    \item \textbf{Binary products}: given games $G = \langle O, \{R_i\}_{i \in I}, \{\preceq _i\}_{i \in I} \rangle$ and $G' = \langle O', \{R'_i\}_{i \in I}, \{\preceq '_i\}_{i \in I} \rangle$, a product in the category \gami\  is a game $G\times_I G'$ such that there exist projection morphisms $\pi_1:G\times_I G'\to G$ and $\pi_2:G\times_I G'\to G'$ and for any game $X$ with morphisms $f:G\times_I G'\to G$ and $g:G\times_I G'\to G'$, there exist a \textit{unique} morphism $u:X\to G\times_I G'$ such that $\pi_1\circ u=f$ and $\pi_2\circ u=g$:\[\xymatrix{
		&X\ar[dl]_f\ar[dr]^g\ar@{-->}^u[d]&\\
G&G\times_IG'\ar[l]_{\pi_1}\ar[r]^{\pi_2}&G' }\]
 We say $G\times_I G'$ is \textit{a} product because there may be other objects in the category with the same property, but it is a standard result that they all must be isomorphic to each other. The notation ``$\times_I$'' uses the letter I as a subscript to make it clear that this is the product in the category \gami.
    
    Such a product can be constructed defining:
	\[G\times_I G'  = \langle  O\times O', \{R^{\times_I}_{{i}}\}_{i \in I}, \{\preceq^{\times_I}_{i}\}_{i \in I} \rangle\]
	where 
	\[ \pair{o,o'}R^{\times_I}_{i} \pair{p,p'}  \text{ iff }  oR_i p   \text{ and } o' R'_ip'\]
	\noindent and
	\[\pair{o,o'}\preceq^{\times_I}_i\pair{p,p'} \text{ iff } o\preceq_i p\text{ and } o'\preceq'_i p'.\]
	$R^{\times_I}_i$ is reflexive and symmetric, since  both $R_i$ and $R'_i$ are reflexive and symmetric relations. By the same token, $\preceq_i^{\times_I}$ is a reflexive and transitive relation.
	
	The natural projection functions $\pi_1:O\times O'\to O$ and $\pi_2:O\times O'\to O'$ provide the projection morphisms. For instance, for each $i\in I$,  if $\pair{o,o'}R^{\times_I}_{i}\pair{p,p'}$ then by definition ,$oR_i p$ that is, $\pi_1\pair{o,o'}R_i \pi_1\pair{p,p'}$. The function $u$ is defined by $u(o)=\pair{f(o),g(o)}$ for each outcome $o$ in the game $X$, and it gives the corresponding game morphism $u$.

The product $G\times_IG'$ can be understood as a game in which the players in $I$ are playing simultaneously the games $G$ and $G'$. 

    \item  \textbf{Small products}:  small products are the generalization of binary products to arbitrary families of games, as long as we only consider a \textit{set} of   games. Let $S$ be a set and $\mathcal{G}=\set{G_s}_{s\in S}$ be such a family, with $G_s= \langle O_s, \{R_{i,s}\}_{i \in I}, \{\preceq _{i,s}\}_{i \in I} \rangle$. One way to obtain a small product is to consider 
	\[\prod_{s\in S} G_s  = \langle  \prod_{s\in S} O_s, \{R^{\Pi_I}_{{i}}\}_{i \in I}, \{\preceq^{\Pi_I}_{i}\}_{i \in I} \rangle,\]
	where, if we write the elements of $\prod_{s\in S} O_s$ as $\bar{o}$, with $\bar{o}_s\in O_s$ for every $s\in S$, then 
     \[ \bar{o}R^{\Pi_I}_{i} \bar{p} \text{ iff }  \bar{o}_s R_{i,s} \bar{p}_s  \text{ for every } s\in S, \]
	 and similarly \[\bar{o}\preceq^{\Pi_I}_i\bar{p} \text{ iff } \bar{o}_s\preceq_{i,s} \bar{p}\text{ for every } s\in S.\]
	As in the binary case, it is easy to check 
 that $R^{\prod_I}_i$ is reflexive and symmetric, since every $R_{i,s}$ is a reflexive and symmetric relation. It also follows that  $\preceq_i^{\prod_I}$ is a reflexive and transitive relation.
 
	The natural projection functions $\pi_s:\prod_{s\in S}O_s\to O_s$  provide the projection morphisms from $\prod_{s\in S} G_s$ to $G_s$. The universal property that this product satisfies is that if there is a game $X$ with morphisms $f_s:X\to G_s$ for every $s\in S$, then there is a unique morphism $u:X\to\prod _{s\in S}G_s$ such that for every $s\in S, \pi_s\circ u=f_s$. 
 \end{itemize}

 In categorical terms, coproducts are the dual of products, meaning that they have a similar structure, but with the arrows reversed. A concrete construction of coproducts can, however, look rather different from the product. In the category \gami\ we have:

\begin{itemize}
\item  \textbf{Binary coproduct}: a binary coproduct of two objects $G$ and $G'$ in a category is an object $G+G'$ together with morphisms $i_G:G\to G+G'$ and $i_{G'}:G'\to G+G'$ sich that for any object $Y$ and morphisms $f:G\to Y$ and $g:G'\to Y$, there exists a unique morphism $v:G+G'\to Y$ such that $v\circ i_G=f$ and $v\circ i_{G'}=g$, as in the following diagram: 
\[\xymatrix{
			G\ar[r]^{i_G}\ar[dr]_{f}	&G+ G'\ar@{-->}[d]^{v}	&G'\ar[l]_{i_{G'}}\ar[dl]^{g}\\
			&  O''					 	& \\
		}
		\] 
 If we denote with $X+Y$ the disjoint union of the sets $X$ and $Y$, and identify the elements in $X$ and $Y$  with their inclusion in $X+Y$, given $G = \langle  O, \{R_i\}_{i \in I}, \{\preceq _i\}_{i \in I} \rangle$ and $G' = \langle  O', \{R'_i\}_{i \in I}, \{\preceq '_i\}_{i \in I} \rangle$, a coproduct arises by forming 
	$$G+_I G'  = \langle  O+ O', \{R_{i}+R'_i\}_{i \in I}, \{\preceq_i+\preceq'_i\}_{i \in I} \rangle,$$
	 where $R_i+R'_i$ is a reflexive and symmetric relation since $R_i$ and $R'_i$  are reflexive and symmetric relations for every $i\in I$. Similarly,  $\preceq_i+\preceq'_i$ is a reflexive and transitive relation, since $\preceq_i$ and $\preceq'_i$ are reflexive and transitive relations. The definitions of these relations ensure that the set inclusions $i_O:O\to O+O'$ and $i_{O'}:O'\to O+O'$ correspond to the inclusion morphisms $i_G:G\to G+_IG'$  and $i_{G'}:G'\to G+_IG'$ respectively. The morphism $v$, also typically indicated by $[f,g]$ is defined by $[f,g](o)=f(o)$ if $o\in O$ and $[f,g](o)=g(o)$ if $o\in O'$.
 
An interpretation of the game $G+_IG'$  is that the agents in the set $I$ are given access to all possible outcomes from both  $G$ and $G'$, meaning they can consider and respond to the full set of consequences arising from either game.

\item \textbf{Small coproducts}: As we saw for small products,  small coproducts are the generalization of binary coproducts to a family of games indexed by a  \textit{set}  $S$: let $\mathcal{G}=\set{G_s}_{s\in S}$, with $G_s= \langle O_s, \{R_{i,s}\}_{i \in I}, \{\preceq _{i,s}\}_{i \in I} \rangle$. We can define a small  coproduct through 
	\[\coprod_{s\in S} G_s  = \langle  \coprod_{s\in S} O_s, \{R^{\coprod_I}_{{i}}\}_{i \in I}, \{\preceq^{\coprod_I}_{i}\}_{i \in I} \rangle.\]
	We can write the elements of $\coprod_{s\in S} O_s$ as pairs of the form $\pair{o_s,s}$, with ${o}_s\in O_s$ for every $s\in S$, and the relations on the coproduct are defined by the disjoint union of the relations:
	\[ R^{\coprod_I}_{{i}}=\coprod_{s\in S}R_{i,s}.\]
    Put another way, we have that 
    \[\pair{o_s,s} R^{\coprod_I}_{{i}} \pair{p_{s'},s'} \text{ iff } s=s' \text{ and } o_sR_{i,s}p_s.\]
     	 and similarly \[\pair{o,s}\preceq^{\coprod_I}_i\pair{p,s'} \text{ iff } s=s' \text{ and } o\preceq_{i,s} p.\]
	 The reflexivity and symmetry of the relations $R^{\coprod_I}_{{i}}$  follow easily from the definition and the fact that each relation $R_{i,s}$ is reflexive and symmetric. Similarly,  $\preceq_i^{\coprod_I}$ is a reflexive and transitive relation, because each $\preceq_{i,s}$ is so. The definitions ensure that the injections $i_s:O\to \coprod_{s\in S}O_s$ given by $i_s(o)=\pair{o,s}$ give the corresponding inclusion morphisms $i_s:G_s\to \coprod_{s\in S}G_s$.
\end{itemize} 

{Besides products, two other limit objects can be defined in \gami, namely  {\em terminal objects} and {\em equalizers}}:
\begin{itemize}
\item \textbf{Terminal object}: let $\mathbb{T}_I = \langle  O_{\mathbb{T}},\{ R^{\mathbb{T}_I }_{i}\}_{i\in I}, \{ \preceq^{\mathbb{T}_I }_{i}\}_{i\in I} \rangle$ be a game in which  $O_{\mathbb{T}}$ is a singleton $O_{\mathbb{T}}=\set{o}$, while all the relations $R^{\mathbb{T}_I }_{i}$ and $\preceq^{\mathbb{T}_I }_{i}$ are the identity on $O_{\mathbb{T}}$, $\set{\pair{o,o}}$. This game has the universal property that given any game  $G$ in \gami, there is a unique morphism from $G$ to $\mathbb{T}_I$, given by the function  $!_{O}: O\to O_{\mathbb{T}}$ that takes all the outcomes of $G$ to the unique one of $\mathbb{T}_I$. 
\end{itemize}
In words: there exists a game $\mathbb{T}_I$ (unique up to isomorphism) in \gami\  such that for each game $G$ there exists a unique morphism from $G$ to $\mathbb{T}_I$. 

 \begin{itemize}
 \item \textbf{Equalizer}: 
 let \( f \) and \( g \) be two morphisms from \( G \) to \( G' \). Then we can define a game  
\[
E = \langle O_E, \{ R^{E_I }_{i} \}_{i\in I}, \{ \preceq^{E_I }_{i} \}_{i\in I} \rangle,
\]  
where  
\( O_E = \{ o \in O \mid f(o) = g(o) \} \) consists of the outcomes on which  \( f \) and \( g \) coincide,  \( R^{E_I }_{i} = R_i|_{O_E} \) is the restriction of \( R_i \) to \( O_E \), and   \( \preceq^{E_I }_{i} = \preceq _{i}|_{O_E} \) is the restriction of \( \preceq_i \) to \( O_E \).  

Since \( R^{E_I }_{i} \) is a restriction of \( R_i \), it remains reflexive and symmetric. Similarly, \( \preceq^{E_I }_{i} \), being a restriction of \( \preceq_i \), remains a preorder for every \( i \in I \). The equalizing morphism \( e: E \to G \) is given by the inclusion map \( e: O_E \to O \).  

Moreover, \( E \) satisfies the universal property of equalizers: for any game \( X \) with a morphism \( h: X \to G \) such that \( f \circ h = g \circ h \), there exists a unique morphism \( u: X \to E \) making the diagram commute, \textit{i.e.},  
$e \circ u = h$. 
This property ensures that \( E \) is the most general game capturing the constraints imposed by the equality \( f = g \) on the outcomes.
\[
\xymatrix{
  X \ar@{-->}[d]_{ u} \ar[dr]^{h}&  \\
  E \ar[r]_{e} & G \ar@<.5ex>[r]^{f} \ar@<-.5ex>[r]_{g} & G'
}
\]

\end{itemize} 
 
Analogously, the corresponding {\em colimits}, i.e. {\em initial objects} and {\em coequalizers}, also exist in \gami:

\begin{itemize}
\item  \textbf{Initial object}: consider the  empty game $\bm{\emptyset}_I$, where the set of outcomes is the empty set, and for each $i\in I$,  all the relations $R_i$ and $\preceq_i$ are the empty set as well. There exists a unique morphism from $\bm{\emptyset}_I$ to any game $G$.
	
\item	 \textbf{Coequalizer}: in the category of sets the coequalizer of two functions $f, g:X\to Y$ is obtained as a quotient by  a relation $\sim$, the smallest equivalence relation on the set $Y$ such that the pairs $\pair{f(x),g(x)}$ are in the relation for every $x\in X$.  If we denote by $[y]$  the elements of $Y/_{\sim}$,  for any relation $R$, and equivalence relation $\sim$ over $Y$, $ R/_{\sim}$ is the relation $\set{\pair{[y],[y']}:\pair{y,y'}\in R}$.

Accordingly, if $f$ and $g$ are two morphisms from $G$ to $G'$ let ${\sim}$ be the equivalence relation on $O'$ generated by the pairs $\pair{f(o),g(o)}$ for all the outcomes $o$ of the game $G$. Thus we can define a game 
		\[G'/_{\sim} =\langle O'/_{\sim},  \{R'_{j}/_\sim\}_{j \in J}, \{(\preceq'_{j}/_\sim)^t\}_{j \in J}\rangle\]
	where  by $(\preceq'_{j}/_\sim)^t$ we denote the transitive closure of the relation $\preceq'_{j}/_\sim$.  
	
	It can be proved that $[\cdot]:G' \rightarrow G'/_{\sim}$ is a morphism in \gami\ and that $G'/_{\sim}$ has the universal mapping property:  $[\cdot]\circ f=[\cdot]\circ g$ and if there is a morphism $m:G'\to G''$ such that $m\circ f=m\circ g$, then there is a unique morphism $u:G'/_{\sim}\to G''$ such that $u\circ[\cdot]=m$.    

\end{itemize}

Recalling that a category is {\em complete} if it has small products and equalizers and it is {\em cocomplete} if it has all the small coproducts and coequalizers we have shown that:

\begin{theorem}~\label{complete1}
	\gami  \  is a complete and cocomplete category.
\end{theorem}

This result indicates that on top of the constructions given above, the category \gami\ supports the construction of limits and colimits of any small diagram. This includes \textit{pullbacks} and \textit{pushouts} (see Example \ref{examplePushout}), and in general provides a strong foundation for working with categorical constructions systematically.

Another construction of interest
is that of {\em exponential object}. The game $G'^G$ has  as outcomes the different ways that the outcomes in  $G$ can be mapped to the outcomes in $G'$  while preserving the accessibility and preference relations. 
We have:

 \begin{proposition}~\label{exp}
    Given games $G, G'$  in \gami, there exists an {\em exponential object} $G'^G$.
\end{proposition}
 
\begin{proof}  Given two games $G = \langle O,\{R_{i}\}_{i \in I}, \{\preceq _{i}\}_{i \in I}\rangle$ and $G' = \langle  O', \{R'_i\}_{i \in I}, \{\preceq '_i\}_{i \in I} \rangle$, we define a game:
		\[G'^{G} = \langle \mathcal{O}, \{R''_{i}\}_{i \in I}, \{\preceq''_{i}\}_{i \in I}\rangle\]
	where 	{$\mathcal{O} =\set{f\in O'^O:f  \ \mbox{is a morphism from } G \text{ to }G'}$}, and the relations $R''_i$ and $\preceq''_i$ are defined by: for all $f, g\in\mathcal{O}$, 
	\[fR''_ig \text{ if and only if for all }o, p\in O, oR_ip \text{ implies }f(o)R_i'g(p),\]
	and
 	{\[ (*) \ \ \ \ \ f\preceq''_ig \text{ if and only if for all }o\in O, f(o)\preceq'_i g(o). \]}
{The condition in the definition of $\preceq''_i$ may seem simpler than the one for $R''_i$, but given that $\preceq_i$ and $\preceq'_i$ are preorders, one can derive a similar condition: if   $o\preceq_i p$, then by ($*$), $f(o)\preceq'_ig(o)$,  and since $g$ is a morphism, $g(o)\preceq'_i g(p)$ so $f(o)\preceq'_ig(p)$ by the transitivity of $\preceq'_i$.}
 
With these definitions, we can prove that  $R''_{i}$ is reflexive and symmetric. Indeed, since for each morphism $f:G\to G'$, every $i\in I$ and every $o,p\in O$, $oR_ip$ implies $f(o)R'_if(p)$, we have that $fR''_if$. For symmetry, assume that $fR''_ig$. Then, if $oR_ip$, by the symmetry of $R_i, pR_io$ so $f(p)R'_ig(o)$ and by the symmetry of $R'_i$, $g(o)R'_if(p)$, which proves $gR''_if$.
	
Similarly, it can be proved that  $\preceq''_{i}$ is a preorder. Reflexivity is immediate, since for all $o \in O$, $f(o) \preceq'_i f(o)$ and thus $f \preceq''_i f$. For transitivity assume that $f \preceq''_i g$ and $g \preceq''_i h$. Then, for all $o \in O$ we have that $f(o) \preceq'_i g(o)$ and $g(o) \preceq'_i h(o)$. By transitivity of $\preceq'_i$ we have $f(o) \preceq'_i h(o)$ for all $o \in O$. That is, $f \preceq''_i h$.

For any other game $G'' = \langle  O'', \{S_i\}_{i \in I}, \{\preceq_i\}_{i \in I}\rangle$, given a morphism $h: G^{''} \times G \rightarrow G'$, we can define as in the category of sets $\psi: G''  \rightarrow G'^G$ by $\psi(o'')(o)=h(o'',o)$.  First, we need to check that for every $o''\in O''$, $\psi(o'')$ is a morphism from $G$ to $G'$, and then that $\psi$ itself is a morphism from $G''$ to $G'^G$.

 To see the first part, fix $o''\in O''$, and  suppose that $o,p\in O$ are such that $oR_ip$. Then, as $\pair{o'',o} (S_i{\times_I}R_i)\pair{o'',p}$, it follows that $h({o'',o}) R'_ih({o'',p})$, that is, $\psi(o'')(o) R'_i \psi(o'')(p)$. 

 To show that $\psi$ is a morphism, take now $o'',p''\in O''$ such that $o''S_ip''$. For any $o, p\in O$ such that $oR_ip$,  we have that $\pair{o'',o} (S_i{\times_I}R_i)\pair{p'',p}$, and thus $h({o'',o}) R'_ih({p'',p})$ so  $\psi(o'')(o) R'_i \psi(p'')(p)$. This implies that $\psi(o'') R''_i \psi(p'')$.
 
 Thus $\psi$ is a morphism that makes the following diagram commute, and its uniqueness follows from its definition on the corresponding sets.
	\[\xymatrix{
		G''\ar@{-->}[d]_{\psi} & 	G'' \times G \ar[d]_{\psi \times id_G} \ar[rrd]^{{h}} & & \\
		G'^{G}	&  G'^{G} \times G \ar[rr]_{eval} &  & G' }\]
 \end{proof}

In game-theoretic terms, the exponential game  \(G'^G \) can be interpreted as a game where the outcomes correspond to \emph{strategic transformations} that is, choices of how to play one game (\( G \)) within another game (\( G' \)).  
 The outcomes of \( G'^G \) are morphisms \( f: G' \to G \), meaning that each possible outcome maps the  structure of \( G \) into \( G' \).  A player choosing an outcome in \(  G'^G \) is effectively selecting a way to see aspects of the game \( G \) reflected in the game \( G' \). The outcomes available to the players  are not just those of a single game but could represent possible manners of determining which aspects of one game are reflected in another.   

In essence, an exponential game represents a space of transformations, allowing players to select how one game is interpreted within another. This can be interpreted as capturing features like \textit{adaptability} and \textit{policy selection}.

	As a consequence of Theorem~\ref{complete1} and Proposition~\ref{exp}:
	
	\begin{theorem}~\label{cartesian}
		\gami \ is a cartesian complete category.
	\end{theorem}
	
\begin{proposition}\label{injectionExp}
    If $G\neq\bm{\emptyset}$, then there is an injection $\psi:G'\to G'^G$.
\end{proposition}
\begin{proof}
    Define for each $o'\in O'$, $\psi(o')=\bm{o'}$, where $\bm{o'}$ is the constant function defined by  $\bm{o'}(o)=o'$ for every  $o\in O$. Clearly $\psi$ is injective. Now, if we have that for some player $i\in I$, $o'R'_ip'$, then $\psi({o'})R''_i\psi({p'})$: indeed, if $o,p\in O$ and $oR_ip$, then $\bm{o'}(o)R'_i\bm{p'}(p)$. A simpler argument works for the preferences. 
\end{proof}

The theorem above means that the original game \( G' \) can be identified with a subset of the  transformations in \( G'^{G} \).  
So, there is a way to reinterpret the outcomes of \( G' \) as particular  mappings from \( G \) to \( G' \). 

This result hints at the richness of the structure of the exponential game, and shows once again the advantages of working in a categorical framework.
 
 A particular case of \gami \ is the category of ``games'' with a single player $\mathbf{1}$, denoted {\bf Gam}$_{\mathbf{1}}$. It is immediate to see that Theorem~\ref{complete1} and Proposition \ref{exp} apply to {\bf Gam}$_{\mathbf{1}}$, where $I = \{\mathbf{1}\}$. In this category the objects represent  {\em decision problems}. The accessibility relation $R_{\mathbf{1}}$ can be understood as the feasibility of changing from a possible outcome to another, while $\preceq_{\mathbf{1}}$ represents the preferences over the outcomes.

 We also have \textbf{Gam}$_\emptyset$, in which the games have no players and therefore no relations are given over the sets of outcomes. Thus, each game in this category can be identified with the set of its outcomes and the morphisms among games in this category with the corresponding set functions. 

 \begin{proposition}
Given a morphism $f:G\to G'$   in \gami, 
\begin{enumerate}
    \item $f$ is monic iff $f:O\to O'$ is injective.
    \item $f$ is epic iff $f:O\to O'$ is surjective.
\end{enumerate}
\end{proposition}

\begin{proof}
\begin{enumerate}
    \item If $f:O\to O'$ is not injective, there exist outcomes $o\neq p\in O$ such that $f(o)=f(p)$. Consider the terminal game $\mathbb{T}_{\bm{1}}$ with a single outcome $*$ in which all the relations $R_i$ and $\preceq_i$ are the identity on the singleton. Let $x$ and $y$ be the functions that send $*$ to  $o$ and $p$ respectively. It is easy to check that $x$ and $y$ are morphisms and $f\circ x=f\circ y$ but $x\neq y$, so $f$ cannot be monic.

    In the other direction, assume that $f$ is injective, and $x$ and $y$ are morphisms from a game $G''$ to $G$ such that $f\circ x=f\circ y$. Then for every outcome $o''$ of $O''$ we have that $f(x(o''))=f(y(o''))$, and since $f$ is injective, $x(o'')=y(o'')$, so $x=y$. 
    \item If $f$ is not surjective, then there exists an outcome $o'$ that is not in the image of $f$. Consider a game $G''$ with two different outcomes, $O''=\set{a,b}$ and the relations $R''_i$ and   $\preceq''_i$ on them are all equal to $O''\times O''$. Now consider the constant morphism $\bm{a}$ and a morphism $g$ that sends all the outcomes in $O'$ to $a$, except for  $o'$ that goes to $b$. Since all elements are related in $G''$, $g$ is  a morphism and different from $\bm{a}$,  with $\bm{a}\circ f=g\circ f$.

    If we assume that $f$ is surjective and $g, h:G'\to G''$ are such that $g\circ f=h\circ f$, then for every $o'\in O'$ we have that $o'=f(o)$ for some $o\in O$. Therefore $g(o')=g(f(o))=h(f(o))=h(o')$ which proves that $g=h$ and therefore $f$ is epic.  \qedhere
\end{enumerate} 
\end{proof}

\subsection{Games with different sets of players}

 We now generalize \gami \ by considering 
 games with variable sets of players. Then, since a given pair of games $G$ and $G'$ may have different sets of players,  $I$ and $J$, a morphism $f: G \rightarrow G'$ should now indicate how $I$ maps onto $J$ and how it preserves the relations defining the games. More precisely:
	
	\begin{definition}~\label{definitionMorphism}
	Each object in the category \gam \ is  $G = \langle I, O, \{R_i\}_{i \in I}, \{\preceq _i\}_{i \in I} \rangle$ as in Definition \ref{definitionGame}. If $G' = \langle J, O', \{R'_j\}_{j \in J}, \{\preceq' _j\}_{j \in J} \rangle$ is also an object, a morphism  $f:G\to G'$   consists of a pair of functions $f_p: I \rightarrow J$ and  $f_O: O\to O'$, such that given $o, p \in O$, for each $i\in I$:
		\begin{itemize}
            \item {if  $oR_ip$ then $ f_O(o)R'_{f_p(i)} f_O(p)$}, and
			\item if  $o \preceq_i p$, then $f_O(o) \preceq_{f_p(i)}' f_O(p)$.
		\end{itemize}
		That is, for each $i\in I$, $f_O:\pair{O,R_i}\to\pair{O',R'_{f_p(i)}}$ and $f_O:\pair{O,\preceq_i}\to\pair{O',\preceq'_{f_p(i)}}$ are {\bf EndoRel} and {\bf PreOrd} morphisms, respectively. We write $f=(f_p,f_O)$.
	\end{definition}
	
	These objects and morphisms define a category $\mathbf{Gam}$, since similarly to the case of \gami \ the composition of morphisms yields a morphism and each $G$ has an associated identity morphism (consisting of the identities on the set of players and on the set of outcomes). Composition is also associative. 

 The following definitions highlight two particular kinds of games:
\begin{definition}
    A \emph{set of players} is a game $G_p(I)=\pair{I, \emptyset, \set{\emptyset}_{i\in I},\set{\emptyset}_{i\in I}}$, with an empty set of outcomes and empty accessibility and preference relations. If the set $I=\set{i}$ is a singleton, we refer to this game as  \emph{player} $G_p(i)$.
\end{definition}

\begin{definition}
    A \emph{set of outcomes} is a game $G_O(O)=\pair{ \emptyset,O, \emptyset,\emptyset}$, with an empty set of players and no accessibility nor preference relations. 
\end{definition}
 We abuse somewhat the terminology by naming these games \textit{players} and \textit{outcomes}, but it is useful to have them as objects in the category \gam. Furthermore, $G_p$ and $G_O$ are easily checked to be functors from the category of sets to  \gam.

\begin{proposition}\label{playerGame}
    For each game $G = \langle I, O, \{R_i\}_{i \in I}, \{\preceq _i\}_{i \in I} \rangle$, $S\subseteq I$, and $T\subseteq O$ there is a unique morphism $f$ from $G_p(S)$ to $G$ such that $f_p=i_S$, the inclusion map $S\hookrightarrow I$ , and also there is a unique $g$ from $G_O(T)$ to $G$ such that $g_O=i_T:T\hookrightarrow O$.
\end{proposition}
\begin{proof}
    The morphisms are $f=(i_S, \emptyset)$ and $g=(\emptyset,i_T)$, respectively.
\end{proof}

	\begin{theorem}~\label{complete}
		$\mathbf{Gam}$ is a complete and cocomplete category.
	\end{theorem}
\begin{proof}
To prove this claim, we must reproduce the constructions from the proof of Theorem~\ref{complete1}, but now in the context of \gam, where they differ from their counterparts in \gami\  due to the interaction between the components \( f_p \) and \( f_O \) of morphisms \( f \).
	
	\begin{itemize}
		\item \textbf{Small products}: given a set-indexed family of games $\mathcal{G} = \set{G_s}_{s\in S}$ with $G_s=\langle I_s, O_s, \{R_{i,s}\}_{i \in I_s}, \{\preceq _{i,s}\}_{i \in I_s} \rangle$, a product can be defined as 
		\[\prod _{s\in S}G_s  = \langle \prod_{s\in S} I_s, \prod_{s\in S}O_s, \{R^{\prod}_{\alpha}\}_{\alpha \in\prod I_s}, \{\preceq^{\prod}_{\alpha}\}_{\alpha \in \prod I_s} \rangle\]
		where for  of $\bar{o},\bar{p}\in \prod_{s\in S} O_s$ and $\alpha \in\prod_{s\in S} I_s$, we define 
     \[ \bar{o}R^{\Pi}_{\alpha} \bar{p} \text{ iff }  \bar{o}_s (R_{\alpha_s,s} )\bar{p}_s  \text{ for every } s\in S, \]
	 and similarly \[\bar{o}\preceq^{\Pi}_\alpha\bar{p} \text{ iff } \bar{o}_s\preceq_{\alpha_s,s} \bar{p}\text{ for every } s\in S.\]
	One can check that $R^{\prod}_i$ is reflexive and symmetric because every $R_{i,s}$ is a reflexive and symmetric relation. It also follows that  $\preceq_i^{\prod_I}$ is a reflexive and transitive relation.\\
	The natural projection functions $\pi_{sp}:\prod_{s\in S} I_s\to I_s$ and $\pi_{sO}:\prod_{s\in S}O_s\to O_s$  provide the projection morphisms from $\prod_{s\in S} G_s$ to $G_s$: if $\bar{o}R^{\Pi}_{\alpha} \bar{p}$ then by definition, $\bar{o}_s (R_{\alpha_s,s} )\bar{p}_s $ but this is to say that $\pi_{sO}(\bar{o}) (R_{\pi_{sp}(\alpha),s} )\pi_{sO}(\bar{p})$.\\
   The product satisfies the universal property that for any game \( X \) equipped with morphisms \( f_s: X \to G_s \) for each \( s \in S \), there exists a unique morphism \( u: X \to \prod_{s \in S} G_s \) such that \( \pi_s\circ u = f_s \) for all \( s \in S \).

				\item \textbf{Terminal object}: let $\mathbb{T} = \langle I_{\mathbb{T}}, O_{\mathbb{T}}, \{R^{\mathbb{T}}\}, \{\preceq^{\mathbb{T}}\} \rangle$ be the game in which $I_{\mathbb{T}}$ and  $O_{\mathbb{T}}$ are the singletons $\set{*_p}$ and $\set{*_O}$ respectively, while $R^{\mathbb{T}}$ and $\preceq^{\mathbb{T}}$, are the identity relation on $O_{\mathbb{T}}$, that is the set $\set{\pair{*_O,*_O}}$. Given any other game $G = \langle I, O, \{R_i\}_{i \in I}, \{\preceq _i\}_{i \in I} \rangle$ there are two unique functions, $!_{Gp}: I \rightarrow I_{\mathbb{T}}$ and  $!_{GO}: O\to O_{\mathbb{T}}$, and they determine the unique  game morphism from  $G$ to $\mathbb{T}$.
		
		\item \textbf{Equalizers}:  let $f$ and $g$ be two morphisms from $G$ to $G'$. Then we can define $E=\langle I_E,O_E,\{R_{i}^{E}\}_{i \in I_E}, \{\preceq _{i}^{E}\}_{i \in I_E}\rangle$, where $I_E=\set{i\in I:f_p(i)=g_p(i)}$, $O_E=\set{o\in O:f_O(o)=g_O(o)}$, and 
 for every  $i\in I_E$, $R_{i}^{E}=R_i|_{O_E}$, and $\preceq _{i}^{E}=\preceq _{i}|_{O_E}$. Since  $R_{i}^{E}$ is a restriction of $R_{i}$, it is also reflexive and symmetric, while $\preceq _{i}^{E}$, being a restriction of $\preceq_i$ it is also a reflexive and transitive relation for each $i \in I$. The equalizing morphism $e:E\to G$ is given by the inclusion maps $e_p:I_E\to I $ and $e_O:O_E\to O$.

		\item \textbf{Coproducts}: Let $Id_X$ be the set of pairs $\pair{x,x}$ for $x\in X$. Given games $G = \langle I, O, \{R_i\}_{i \in I}, \{\preceq _i\}_{i \in I} \rangle$ and $G' = \langle J, O', \{R'_j\}_{j \in J}, \{\preceq '_j\}_{j \in J} \rangle$, a natural way to derive a coproduct is to use 
		$$G+ G'  = \langle I+ J, O+ O', \{R^+_{k}\}_{k \in I+ J}, \{\preceq ^+_{k}\}_{k \in I+ J} \rangle$$
		where 
		\[R^+_i=R_i\cup Id_{O'} \text{ if } i\in I, \text{ and } R^+_j=R'_j\cup Id_O,\text{ if } j\in J\]
		and similarly, 
		\[\preceq^+_i=\preceq_i\cup Id_{O'} \text{ if } i\in I, \text{ and } \preceq^+_j=\preceq'_j\cup Id_O,\text{ if } j\in J.\]
		
        Here we need to add the identity relations to make sure both $R_k^{+}$ and $\preceq_k^+$ are reflexive. Thus, $R_k^{+}$ is a reflexive and symmetric relation since $R_i$ and $R_j$  for $i \in I$ and $j \in J$, as well as $Id_O$ and $Id_{O'}$ are reflexive and symmetric relations. In turn, and $\preceq_k^{+}$ is a reflexive and transitive relation, since $\preceq_i$, $\preceq_j$, $Id_O$ and $Id_{O'}$ are reflexive and transitive relations.
		
  Let $i_O$ and $i_{O'}$	be the injections from $O$ and $O'$ into $O+O'$ respectively.	To see that adding the identity relations on the outcomes preserves the universal property of the coproduct, consider morphisms $f:G\to G''$ and $g:G'\to G''$. We can build the set functions $[f,g]_p$ and $[f,g]_O$ as usual. If $\pair{x,y}\in R_k^+$, with $k\in I$, we consider two cases: 
  \begin{itemize}
 \item  $x,y\in O$, then $[f,g]_O(i_O(x))=f_O(x)$  and $[f,g]_O(i_O(y))=f_O(y)$ so, since $f$ is a morphism, $\pair{f_O(x),f_O(y)}\in R''_{f_p(k)}$.
 \item  $x,y\in O'$ then  we must have $x=y$ so $[f,g]_O(i_{O'}(x))=g_O(x)=g_O(y)=[f,g]_O(i_{O'}(y))$. In this case, we know that since $R''_{f_p(k)}$ is reflexive, $\pair{g_O(x),g_O(y)}\in R''_{f_p(k)}$.
\end{itemize}
 \[\xymatrix{
			O\ar[r]^{i_O}\ar[dr]_{f_O}	&O+ O'\ar@{-->}[d]^{[f,g]_O}	&O'\ar[l]_{i_{O'}}\ar[dl]^{g_O}\\
			&  O''					 	& \\
		}
		\] 
  Thus $i_G=(i_I,i_O):G\to G+G'$ and  $i_{G'}=(i_J,i_{O'}):G'\to G+G"$ are the inclusion morphisms, where $i_I$ and $i_J$ are the natural inclusion functions into $I+J$.
		\item \textbf{Small coproducts}: For a family of games indexed by a  \textit{set}  $S$: let $\mathcal{G}=\set{G_s}_{s\in S}$, with $G_s= \langle I_s, O_s, \{R_{i,s}\}_{i \in I}, \{\preceq _{i,s}\}_{i \in I} \rangle$. It is possible to construct a small coproduct by employing 
	\[\coprod_{s\in S} G_s  = \langle  \coprod_{s\in S} I_s,\coprod_{s\in S} O_s, \{R^{\coprod}_{\pair{i,s}}\}_{\pair{i,s} \in\coprod I}, \{\preceq^{\coprod_I}_{\pair{i,s}}\}_{\pair{i,s} \in \coprod I} \rangle.\]
	The players of this game are all the players of the individual games $G_s$. We can see them as elements of $\coprod_{s\in S} I_s$, of the form $\pair{i,s}$, with ${i}\in I_s$ for every $s\in S$, and the 
 corresponding relations on the coproduct are defined on the outcomes of the rest of the games as the minimal reflexive relation:
	\[ R^{\coprod}_{\pair{i,s}}=R_{i,s}+\coprod_{s'\neq s}Id_{O_{s'}}.\]
    Put another way, we have that 
    \[\pair{o,s} R^{\coprod}_{\pair{i,s''} }\pair{p,s'} \text{ iff } (s=s'=s'' \text{ and } oR_{i,s}p )\text{ or }(s=s'\text{ and } o=p),\]
     	 and similarly \[\pair{o,s}\preceq^{\coprod}_{\pair{i,s''}}\pair{p,s'} \text{ iff } (s=s'=s'' \text{ and } o\preceq_{i,s}p )\text{ or }(s=s'\text{ and } o=p).\]
 The definitions given above imply that the injections $i_{sp}:I_s\to \coprod_{s\in S}I_s$ given by $i_{sp}(i)=\pair{i,s}$ and $i_{sO}:O_s\to \coprod_{s\in S}O_s$ given by $i_{sO}(o)=\pair{o,s}$ give the corresponding inclusion morphisms $i_s=(i_{sp},i_{sO}):G_s\to \coprod_{s'\in S}G_{s'}$.
 
 The game $\coprod_{s\in S} G_s$ has the universal property that for any game $Y$, and morphisms $f_s:G_s\to Y$, there is a unique morphism $v:\coprod_{s\in S} G_s\to Y$ such that for every $s\in S$, $v\circ i_s=f_s$. 
		\item \textbf{Initial object}: {\it consider the empty game $\bm{\emptyset}$, where the set of players, the set of outcomes, and the corresponding families of relations are all the empty set. There exists a unique morphism from $\bm{\emptyset}$ to any game $G$.  Notice that $\bm{\emptyset}=G_p(\emptyset)=G_O(\emptyset)$.}
   
		\item \textbf{Coequalizers}: {\it if $f$ and $g$ are two morphisms from $G$ to $G'$ we must now consider an equivalence relation on the set $J$ of players of the game $G'$, generated by the pairs $\pair{f_p(i),g_p(i)}$ for all $i\in I$. We call this equivalence relation $\sim_p$. At the same time, we have the equivalence relation on $O'$ that we defined for coequalizers in \gami, which we now denote by $\sim_O$. The game 
				\[G'/_{\sim} =\langle
		J/_{\sim _p}, O'/_{\sim_O},  \{R'_{[j]}\}_{[j] \in J/_{\sim _p}}, \{\preceq'_{[j]}\}_{[j] \in J/_{\sim _p}}\rangle\]		
		where   $[{o}'] R'_{[j]} [p']$ if and only if there exist some $k \in  [j]$,   $q \in [o']$ and $r \in [p']$ such that  $q R'_{k} r$. Then, $R'_{[j]}$ is reflexive and symmetric, since $R'_{j}$ is reflexive and symmetric for every $j \in J$.\\ 
		
		In turn, for each $[j]$, $\preceq_{[j]}$ is defined as the smallest reflexive transitive relation on $O'/_{\sim}$ that contains $\preceq_k/{\sim_O}$ for each $k\in[j]$. The required morphism is $h=([\cdot],[\cdot]):G'\to G'/_{\sim}$. It is easy to check that $h$ preserves the relations $R'_{j}$ and $\preceq'_j$ for all $j\in J$.}
				 \qedhere
	\end{itemize}
 \end{proof}

 As a particular case we may remark on the construction of binary products:  given games $G = \langle I, O, \{R_i\}_{i \in I}, \{\preceq _i\}_{i \in I} \rangle$ and $G' = \langle J, O', \{R'_j\}_{j \in J}, \{\preceq '_j\}_{j \in J} \rangle$, to form a we proceed to let 
		\[G\times G'  = \langle I\times J, O\times O', \{R_{\pair{i,j}}\}_{\pair{i,j} \in I\times J}, \{\preceq _{\pair{i,j}}\}_{\pair{i,j} \in I\times J} \rangle\]
		where 
		\[\pair{o,o'} R_{\pair{i,j}}\pair{p,p'}  \text{ iff } o R_i p   \text{ and }  o' R'_j p' \]
		\noindent and
        \[\pair{o,o'} \preceq_{\pair{i,j}}\pair{p,p'} \text{ iff } o\preceq_i p\text{ and } o'\preceq_j p'.\]

 Each pair $\pair{i, j}$ can be seen as a coalition that must \textit{jointly evaluate} outcomes based on their respective preferences.  Since preferences are combined using conjunction, an outcome is preferred only if both players in the pair agree on it. 
 
\begin{example}
Consider a company where different departments, e.g. \textit{Finance} (F) and \textit{Marketing} (M) must jointly approve investment projects.  
 Each project has an outcome pair \(\pair{o_F, o_M}\), where \( o_F \) is how Finance evaluates the project and \( o_M \) is how Marketing evaluates it.   A project is only deemed preferable if \emph{both} departments find their respective components better than an alternative.  

This interpretation highlights how product games enforce coordinated decision-making through strict agreement conditions. 
\end{example}

\begin{definition}\label{defF_I} For each set of players $I$ we define a functor $F_I:\gam_I\to\gam$ by putting $F_I(\pair{O, \{R_i\}_{i \in I}, \{\preceq_i\}_{i \in I}})=\pair{I,O, \{R_i\}_{i \in I}, \{\preceq_i\}_{i \in I}}$ on the objects and $F_I(f)=(1_I,f)$ on the morphisms. 
\end{definition}

Notice that this functor does not preserve products or coproducts. For example, a product of games in \gami\ will have $I$ as its sets of players, but when performed in \gam, the number of players will be the square of the number of players in $I$. 

\begin{proposition}~\label{expGam}
	In $\mathbf{Gam}$, given games $G, G'$ there exists an {\em exponential object} $G'^G$. 
\end{proposition}

\begin{proof}: Given two games $G = \langle I,O,\{R_{i}\}_{i \in I}, \{\preceq _{i}\}_{i \in I}\rangle$ and $G' = \langle J, O', \{R'_j\}_{j \in J}, \{\preceq '_j\}_{j \in J} \rangle$, we define a game:
$$G'^{G} = \langle K, \mathcal{O}, \{R_{\kappa}\}_{\kappa \in K}, \{\preceq _{\kappa}\}_{\kappa \in K}\rangle$$
where 
$$K=\{\kappa : I \rightarrow J| \mbox{ there exists } \rho:O\to O' \text{ such that }\ (\kappa, \rho) \in {Hom}_{\mathbf{Gam}}(G, G')\} $$
and 
\[\mathcal{O} = \{\rho: O \rightarrow O' | \mbox{ there exists } \kappa:I\to J \text{ such that } 
(\kappa, \rho) \in {Hom}_{\mathbf{Gam}}(G, G')\}.\]

For each $\kappa \in K$,
\[{R}_{\kappa}= \{(\rho, \rho'): (\kappa, \rho), (\kappa,\rho') \in \mbox{Hom}_{\mathbf{Gam}}(G, G') \ \mbox{and for all} \ o, p \in O \ \mbox{and for all} \ i \in I\]
\[ oR_ip \text{ implies } \rho(o) R'_{\kappa(i)} \rho'(p) \} \cup Id_{\mathcal{O}}.\]

With this specification, $R_{\kappa}$ is trivially reflexive and we can check that it is symmetric as well:
Suppose that $(\rho, \rho') \in {R}_{\kappa}$. By definition, for all $o,p\in O$ and $i\in I$, if $oR_ip$, then, since $R_i$ is symmetric, $pR_io$ and thus $\rho(p)R'_{\kappa(i)}\rho'(o)$. By the symmetry of $R'_{\kappa(i)}$, $\rho'(o)R'_{\kappa(i)}\rho(p)$, and this proves that $\rho'{R}_{\kappa} \rho$.

In the case of the preference relations, for each $\kappa \in K$, $\rho\preceq_\kappa \rho'$ if and only if for all $o \in O$ and  $i \in I$, $\rho(o) \preceq_{\kappa(i)} \rho'(o)$.

While the reflexivity of $\preceq_{\kappa}$ is trivially satisfied, transitivity can be proven as follows. Assume that $\rho \preceq_{\kappa} \rho'$ and $\rho' \preceq_{\kappa} \rho''$. This means that for all $i \in I$ and all $o \in O$, $\rho(o) \preceq_{\kappa(i)} \rho'(o)$ and $\rho'(o) \preceq_{\kappa(i)} \rho''(o)$. Since $\preceq_{\kappa(i)}$ is transitive for each $i \in I$, it follows that $\rho \preceq_{\kappa} \rho''$. Thus, $\preceq_{\kappa}$ is a preorder.

{For any other game $G'' = \langle L, O'', \{S_l\}_{l \in L}, \{\preceq_l\}_{l \in L}\rangle$, given a morphism $h: G^{''} \times G \rightarrow G'$, we  define maps  $\Psi_p$ and $\Psi_O$  as follows:  given $l\in L$, $i \in I$, $o''\in O''$, and $o \in O$,}
\begin{itemize}
\item {$\Psi_p(l)(i) =h_p(l,i) $}
\item {$\Psi_O(o'')(o) = h_O(o'',o)$,}
\end{itemize}

{We need to check that for every $l \in L$ and $o''\in O''$, $(\Psi_p(l), \Psi_O(o''))$ is a morphism from $G$ to $G'$, and then that $\Psi=(\Psi_p,\Psi_O)$ itself is a morphism from $G''$ to $G'^G$.}

To see the first part, fix $l \in L$ and $o''\in O''$, and suppose that $o,p\in O$ are such that $oR_ip$. Then, as $\pair{o'',o} (S_l{\times}R_i)\pair{o'',p}$, it follows that $h({o'',o}) R'_{h_p(l,i)}h({o'',p})$, that is, $\Psi_{O}(o'')(o) R'_{\Psi_p(l)(i)} \Psi_{O}(o'')(p)$.  Similarly, we can check that the preferences are preserved.

{To show that $\Psi$ is a morphism, take $l \in L$ and $o'',p''\in O''$ such that $o''S_lp''$. For any $o, p\in O$ satisfying $oR_ip$,  we have that $\pair{o'',o} (S_l{\times}R_i)\pair{p'',p}$, and thus $h({o'',o}) R'_{h_p(l,i)} h({p'',p})$ so  $\Psi_O(o'')(o) R'_{\Psi_p(l)(i)} \Psi_O(p'')(p)$. This implies that $\Psi_O(o'') R_{\Psi_p(l)} \Psi_O(p'')$. }
 
Thus $\Psi$ is a morphism that makes the following diagram commute, and its uniqueness follows from its definition on the corresponding sets.
$$\xymatrix{
	G''\ar@{-->}[d]_{\Psi} & 	G'' \times G \ar[d]_{\Psi \times id_G} \ar[rrd]^{\mathbf{h}} & & \\
	G'^{G}	&  G'^{G} \times G \ar[rr]_{eval} &  & G' }$$
\qedhere
\end{proof}

Analogously to what we saw in Proposition \ref{injectionExp}, we have the following result:

 \begin{proposition}\label{embeddingGam}
    If $G\neq\bm{\emptyset}$, then there is an embedding $\psi$ of $G'$ in $G'^G$.
\end{proposition}

\begin{proof}
Let  $\psi_p:J\to J^I$ and $\psi_O:O'\to O'^O$, where for every $i\in I$  and $j\in J$, $\psi_p (j)(i)=j$, and  for every $o\in O$  and $o'\in O'$, $\psi_O (o')(o)=o'$. We denote with $\bm{j}$ and $\bm{o'}$ these constant functions.  Furthermore, notice that for any $\kappa:I\to J$, the pair $\pair{\kappa,\bm{o'}}$ is a morphism from $G$ to $G'$. Indeed, if $oR_ip$, then $\bm{o'}(o)R'_{\kappa(i)}\bm{o'}(p)$ for any $o,p\in O$ and $i\in I$.  This proves that in particular $(\bm{j},\bm{o}')$ is a morphism, so $\bm{j}\in K$ and $\bm{o'}\in \mathcal{O}$. 

Finally, we can show that $\psi=(\psi_p,\psi_O)$ is a morphism. Consider $o',p'\in O'$ such that $o'R'_jp'$ for some $j\in J$. We want to show that $\psi_O(o')R_{\psi_p(j)}\psi_O(p')$. By definition, this means that $\pair{\bm{o'},\bm{p'}}\in R_{\bm{j}}$. We already know that  $(\bm{j},\bm{o'})$ and $(\bm{j},\bm{p'})$ are morphisms from $G$ to $G'$, and for all $i\in I$ and $o,p\in O$, $\pair{o,p}\in R_i$ it should be the case that  $\bm{o'}(o)R'_{\bm{j}(i)}\bm{p'}(p)$, but this is equivalent to the hypothesis $o'R'_jp'$.
\end{proof}

The following proposition illustrates the flexibility and expressivity of this formalism: 

\begin{proposition}
    Each set function $f:I\to J$ induces a functor $F_f:\gam_I\to\gam_J$ such that for every set of outcomes $O$ of a game in $\gam_I$, $(f,1_O)$ is a morphism in \gam.
\end{proposition}
\begin{proof}
    For each $G=\pair{O, \{R_i\}_{i \in I}, \{\preceq_i\}_{i \in I}}\in\gam_I$, we put $F_f(G)=\pair{ O, \set{R^f_j}_{j\in J},\set{\preceq^f_j}_{j\in J}}\in\gam_J$, where \[oR^f_jp \text{ iff there exists }i\in I \text{ such that } f(i)=j \text{ and }oR_ip.\]
    For a morphism $g:G\to G'$ in \gam$_I$, $F_f(g):F_f(G)\to F_f(G')$ is given by the same function $g:O\to O'$. We can check that $F_f(g)$ is a morphism: if $oR^f_jp$, then for some  $i\in I$ such that $f(i)=j$,  $oR_ip$. Therefore, $g(o)R'_ig(p) $, so $g(o)R'^f_jg(p)$.  It follows from the definition that $F_f$ preserves compositions and identities.

    If we now consider a game $G$ in \gami, we have that $F_f(G)$ is in \gam$_J$, and we can apply the functors $F_I$ and $F_J$ (see Definition \ref{defF_I}), respectively. The pair $(f,1_O):F_I(G)\to F_J(F_f(G))$ is a morphism in \gam: if $oR_ip$, then $oR^f_{f(i)}p$.
\end{proof}

It is well known that if a category has coproducts and coequalizers, then it has pushouts. The following example shows an application of this construction.

\begin{example}\label{examplePushout}
    Consider these two games $G=\pair{\set{1,2}, O,\set{R_1,R_2},\set{\preceq_1,\preceq_2}}$ and $G'=\pair{\set{2,3}, O',\set{R'_2,R'_3},\set{\preceq'_2,\preceq'_3}}$. Since player $2$ participates in both games, we want to amalgamate $G$ and $G'$ in a way that reflects this situation. For this, we use the player game $G_p(2)=\pair{\set{2}, \emptyset, \set{\emptyset},\set{\emptyset}}$, and the morphisms $f$ from $G_p(2)$ to $G$ and $f'$ from $G_p(2)$ to $G'$ given by proposition \ref{playerGame}. 

    The pushout of $G$ and $G'$ is obtained by building their coproduct and then finding the quotient under the relation $\sim_2$ that identifies player 2 in each game.
\[
\xymatrix{
 &G_p(2)\ar[dl]_f\ar[dr]^{f'}&\\
G\ar[dr]_{i_G}&&G'\ar[dl]^{i_{G'}}\\
&G+G'\ar[d]&\\
&G+G'/_{\sim_2}&
}
\]
The set of players of $G+G'$ can be written as $\set{\pair{1,0},\pair{2,0},\pair{2,1},\pair{3,1}}$ (the second component in each pair indicates to which game corresponds to: $0$ for $G$ and $1$ for $G'$). The relation $\sim_2$ identifies $\pair{2,0}$ and $\pair{2,1}$.

We can now describe the components of this pushout $G+G'/_{\sim_2}$. The set of players is composed by the equivalence classes $[\pair{1,0}]$, $[\pair{2,0},\pair{2,1}]$, $[\pair{3,1}]$.  The set of outcomes is $O+O'$ since there is no restriction imposed by the morphisms from $G_p(2)$.
The  accessibility relations are: $R^+_{[\pair{1,0}]}=R^+_{\pair{1,0}}=R_1\cup Id_{O'}$,
$R^+_{[\pair{3,1}]}=R^+_{\pair{3,1}}=Id_{O}\cup R'_3$, and for $R^+_{[\pair{2,0},\pair{2,1}]}$ we have that $\pair{x,y}\in R^+_{[\pair{2,0},\pair{2,1}]}$ if and only if $\pair{x,y}\in R^+_{\pair{2,0}}$ or $\pair{x,y}\in R^+_{\pair{2,1}}$. That is, if and only if $\pair{x,y}\in R_2+R'_2$.

Analogously, $\preceq^+_{[\pair{1,0}]}=\preceq_1\cup Id_{O'}$, $\preceq^+_{[\pair{3,1}]}=Id_{O}\cup \preceq'_3$, and $\preceq^+_{[\pair{2,0},\pair{2,1}]}=\preceq_2+\preceq'_2$. 

Notice that in the new game players $1$ and $3$ keep their original preferences and accessibility relations on the outcomes of their games, and have trivial ones over those of the other one. On the other hand, player $2$ extends their relations to include both of the games in which they participated.
    \end{example}

\section{Subgames and equilibria}~\label{equilibria}

As usual, we can say that $G$ is a subgame of a game $G'$ if the inclusion map is a morphism in the category \gam, or \gami.\footnote{Notice that this is {\em not} the meaning of {\em subgame} in the Game Theory literature (see \cite{osborne94course}).} We can spell out what this means in detail:

\begin{definition}
    In \gami, a  game $G=\pair{O, \{R_i\}_{i \in I}, \{\preceq_i\}_{i \in I}}$ is a subgame of $G'=\pair{O', \{R'_i\}_{i \in I}, \{\preceq'_i\}_{i \in I}}$ if $O\subseteq O'$ and for each $i\in I$, $R_i$ is a subset of the restriction of $R'_i$ to $O$, while each $\preceq_i$ is a preorder on $O$ that is included in $\preceq'_i$. 

     In \gam, a  game $G=\pair{I,O, \{R_i\}_{i \in I}, \{\preceq_i\}_{i \in I}}$ is a subgame of $G'=\pair{J,O', \{R'_j\}_{j \in J}, \{\preceq'_j\}_{j \in J}}$ if $I\subseteq J$, $O\subseteq O'$ and for each $i\in I$, $R_i$ is a subset of the restriction of $R'_i$ to $O$, while each $\preceq_i$ is a preorder on $O$ that is included in $\preceq'_i$. 
\end{definition}

	One of the main goals of Game Theory is to postulate {\em solution concepts} for games, indicating what results should be expected if players exhibit different ways of making decisions \cite{osborne94course}.
 In the case of games represented by multi-graphs we capture the notion of a solution as a selection of a subgame generated by the set of outcomes on which we keep all the relations of accessibility and preferences among those outcomes: 
	
	\begin{definition}
		A {\em solution concept} is a mapping $\phi$ such that given a game $G = \langle I, O, \{R_i\}_{i \in I}, \{\preceq_i\}_{i \in I} \rangle$ yields a game 
		$$\phi(G) = \langle I, O_{\phi}, \{R^{\phi}_i\}_{i \in I}, \{\preceq^{\phi}_i\}_{i \in I} \rangle$$
 such that $O_{\phi} \subseteq O$ and for each $i \in I$, $R^{\phi}_i$ and $\preceq^{\phi}_i$ are the restrictions of $R_i$ and $\preceq_i$ to $O_{\phi}$.
	\end{definition}
Notice that the set of players of $\phi(G)$ is the same as  in $G$, so this definition works both in \gami\ and \gam. The game $\phi(G)$ is the graph induced by the selected outcomes. The definition of $\phi$ captures which outcomes can be recommended (under some criterion).

 The literature presents several alternative solution concepts, but the most widely used one is {\em Nash equilibrium}, which we adapt to our representation of games:
 
\begin{definition}
      An outcome $o^* \in O$  is a \emph{Nash equilibrium} if  for each $i$, and for each $p$ such that $\langle o^*, p \rangle \in R_i$, $p \preceq_i o^*$. 

      For each game $G=\pair{I,O, \{R_i\}_{i \in I}, \{\preceq_i\}_{i \in I}}$, then, $\phi^{NE}(G)$ is the subgame of $G$    with the same set of players $I$, the set of outcomes $O_{\phi^{NE}}\subseteq O$ of all the Nash equilibria in the game, and the restrictions of $R_i$ and $\preceq_i$ to $O_{\phi^{NE}}$. 

      For simplicity, we will write $o^*\in\phi^{NE}(G)$ instead of $o^*\in O_{\phi^{NE}(G)}$.
\end{definition}

 The concept of a Nash equilibrium captures a stability condition where no player has an incentive to deviate, assuming others do not change their choices. The definition of a weak Nash equilibrium relaxes this condition by allowing for situations where a player may be indifferent about switching  to an alternative outcome. While every Nash equilibrium is a weak Nash equilibrium, the reverse is not necessarily true.

\begin{definition}
      An outcome $o^* \in O$  is a \emph{weak Nash equilibrium} if, for each player $i$, whenever there exists an alternative $p$ satisfying $\langle o^*, p \rangle \in R_i$, it holds that either $p \preceq_i o^*$ or $o^*$ is not  dominated by $p$ according to $\preceq_i$.

      Given a game $G=\pair{I,O, \{R_i\}_{i \in I}, \{\preceq_i\}_{i \in I}}$, we define $\phi^{wNE}(G)$ as the subgame of $G$ where the set of players remains $I$, the set of outcomes $O_{\phi^{wNE}}$ consists of all weak Nash equilibria, and the relations $R_i$ and $\preceq_i$ are the corresponding restrictions to the set $O_{\phi^{wNE}}$. 
\end{definition}

In the case of {\bf Gam}$_{\mathbf{1}}$, the notion of Nash equilibrium for a decision problem consists of a subgame in which each outcome $o^* \in O_{\phi}$ is an {\em optimal solution}. That is, for every $p \in O$, if $o^* R_\mathbf{1} p$, then $p \preceq_\mathbf{1} o^*$. An optimal solution is actually a maximal element of $\preceq_\mathbf{1}$ among those accessible through $R_\mathbf{1}$.

\begin{example}~\label{BoS} 		Consider a game $G_{BoS}$ which can be seen as a multi-graph representation of the following strategic game (known as the {\it Battle of the Sexes}):
		
			\begin{center}
			\begin{tabular}{l|c|c}
				&   C   &   D   \\ \hline
				A & (2,1) & (0,0) \\ \hline
				B & (0,0) & (1,2)
			\end{tabular}
		\end{center}
		 in $G_{BoS}$, $I = \{1, 2\}$ and $O = \{(A,C), (B,D), o\}$, where, as in the second part of Example~\ref{off}, $o$ is an outcome that identifies $(A,D)$ and $(B,C)$. The accessibility and preference relations are represented in Figure~\ref{figNT}. Notice the difference with Figure~\ref{figNT2}, in which the preferences of both players (represented by dashed curves) go in {\em different} directions between $(D,L)$ and $o$ and $(T,R)$ and $o$. Here, instead, the preferences of both players are the {\it same} between $o$ and $(B,D)$ and between $o$ and $(A,C)$.

		We can see (recall that the accessibility relations are not transitive) that both $(A,C), (B,D) \in {\phi^{NE}(G_{BoS})}$.
		
			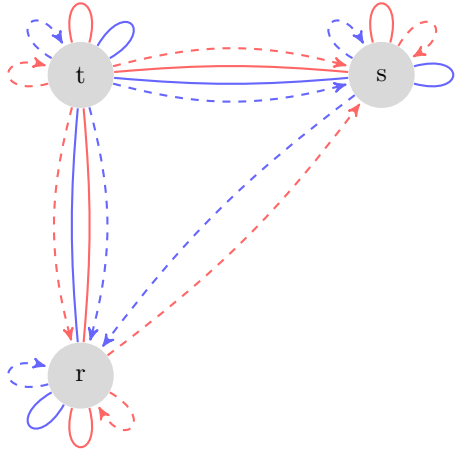
\begin{figure}[h!]
			\centering
			
			\begin{tikzpicture}[scale=1,
			DashedR/.style={
				dashed,
				->,>=stealth',
				shorten >=1pt,
				auto,
				thick,
				draw=red!60},
			DashedB/.style={
				dashed,
				->,>=stealth',
				shorten >=1pt,
				auto,
				thick,
				draw=blue!60},
			FillR/.style={
				thick,
				draw=red!60},
			FillB/.style={
				thick,
				draw=blue!60}]
			
			\tikzstyle{every state}=[fill=gray!30!white,draw=none]
			
			\coordinate (A) at (0,0);
			\coordinate (C) at (4,4);
			\coordinate (D) at (0,4);
			
			\node[state] (N1) at (A) {AC};
			\node[state] (N3) at (C) {BD};
			\node[state] (N4) at (D) {o};
			
			\path
			(N1) edge [FillB, bend left = 5] (N4)
			(N4) edge [FillR, bend left = 5] (N1)
			(N4) edge [DashedR, bend right = 15] (N1)
			(N4) edge [DashedB, bend right = -15] (N1)
			
			(N3) edge [FillB, bend left = 5] (N4)
			(N3) edge [FillR, bend left = -5] (N4)
			(N4) edge [DashedR, bend left = 15] (N3)
			(N4) edge [DashedB, bend left = -15] (N3)
			
			(N1) edge [DashedR, bend right = 8] (N3)
			(N3) edge [DashedB, bend right = 8] (N1);
			
			\draw[DashedB](N1) to [out=195,in=165,looseness=8] (N1);
			\draw[FillB](N1) to [out=240,in=210,looseness=8] (N1);
			\draw[FillR](N1) to [out=285,in=255,looseness=8] (N1);
			\draw[DashedR](N1) to [out=330,in=300,looseness=8] (N1);
			
			\draw[DashedR](N3) to [out=60,in=30,looseness=8] (N3);
			\draw[FillR](N3) to [out=105,in=75,looseness=8] (N3);
			\draw[DashedB](N3) to [out=150,in=120,looseness=8] (N3);
			\draw[FillB](N3) to [out=375,in=345,looseness=8] (N3);
			
			\draw[FillB](N4) to [out=60,in=30,looseness=8] (N4);
			\draw[FillR](N4) to [out=105,in=75,looseness=8] (N4);
			\draw[DashedB](N4) to [out=150,in=120,looseness=8] (N4);
			\draw[DashedR] (N4) to [out=195,in=165,looseness=8] (N4);
			
			\end{tikzpicture}
			
			\caption{Multi-graph of $G_{BoS}$. Blue and red lines correspond to players 1 and 2, respectively. Full lines correspond to accessibility relations and dashed ones to preferences.}
			\label{figNT}
		\end{figure}

	\end{example}

\begin{example}
Consider a game represented by the multi-graph in the figure below. 
\begin{minipage}{7cm}
			\begin{tikzpicture}[scale=1,
			DashedR/.style={
				dashed,
				->,>=stealth',
				shorten >=1pt,
				auto,
				thick,
				draw=red!60},
                DashedBB/.style={
				dashed,
				<->,>=stealth',
				shorten >=1pt,
				auto,
				thick,
				draw=blue!60},
			DashedB/.style={
				dashed,
				->,>=stealth',
				shorten >=1pt,
				auto,
				thick,
				draw=blue!60},
			FillR/.style={
				thick,
				draw=red!60},
			FillB/.style={
				thick,
				draw=blue!60}]
			
			\tikzstyle{every state}=[fill=gray!30!white,draw=none]
			
			\coordinate (A) at (0,0);
			\coordinate (C) at (4,4);
			\coordinate (D) at (0,4);
			
			\node[state] (N1) at (A) {q};
			\node[state] (N3) at (C) {o};
			\node[state] (N4) at (D) {p};
			
			\path
			(N1) edge [FillR, bend right = 0] (N4)
			
			(N3) edge [FillB, bend left = 0] (N4)
			(N4) edge [DashedBB, bend left = 15] (N3)
			
			(N1) edge [DashedR, bend right = 8] (N3)
                (N1) edge [FillR, bend right = 3] (N3)
                (N1) edge [FillB, bend left = 3] (N3) 
			(N1) edge [DashedB, bend left = 8] (N3);
			
			\draw[DashedB](N1) to [out=195,in=165,looseness=8] (N1);
			\draw[FillB](N1) to [out=240,in=210,looseness=8] (N1);
			\draw[FillR](N1) to [out=285,in=255,looseness=8] (N1);
			\draw[DashedR](N1) to [out=330,in=300,looseness=8] (N1);
			
			\draw[DashedR](N3) to [out=60,in=30,looseness=8] (N3);
			\draw[FillR](N3) to [out=105,in=75,looseness=8] (N3);
			\draw[DashedB](N3) to [out=150,in=120,looseness=8] (N3);
			\draw[FillB](N3) to [out=375,in=345,looseness=8] (N3);
			
			\draw[FillB](N4) to [out=60,in=30,looseness=8] (N4);
			\draw[FillR](N4) to [out=105,in=75,looseness=8] (N4);
			\draw[DashedB](N4) to [out=150,in=120,looseness=8] (N4);
			\draw[DashedR] (N4) to [out=195,in=165,looseness=8] (N4);
			
			\end{tikzpicture}
\end{minipage}
\begin{minipage}{7cm}
Outcome $o$ is a Nash Equilibrium, and thus a weak Nash Equilibrium: 
\begin{itemize}
\item[-] $o R_1 p$ and $p \preceq_1 o$, 

\item[-] $o R_1 q$ and $q \preceq_1 o$, 

\item[-] $o R_2 q$ and $q \preceq_2 o$.
\end{itemize}

In turn, $p$ is a weak Nash Equilibrium but {\em not} a (strict) Nash Equilibrium:
\begin{itemize}
\item[-] $p R_1 o$ and $o \preceq_1 p$,
\item[-] $p R_2 q$ while $p \not\preceq_2 q$ and $q \not\preceq_2 p$.
\end{itemize}
\end{minipage}
\end{example}

While game morphisms typically do not preserve equilibria, it is possible to identify equilibria in products, coproducts, and exponents by building upon the equilibria of the original games.

\begin{theorem}
    Given games $G$ and $G'$ in \gami, with $I\neq \emptyset$, then
    \begin{enumerate}
        \item $o^*\in\phi^{NE}(G)$ and $o'^*\in\phi^{NE}(G')$ if and only if $\pair{o^*,o'^*}\in\phi^{NE}(G\times_I G')$.
        \item  $o^*\in\phi^{NE}(G)$   if and only if $i_G(o^*)\in \phi^{NE}(G+_I G')$. A similar result applies to equilibria in  $\phi^{NE}(G')$ 
        \item $o'^*\in\phi^{NE}(G')$ if and only if $\bm{o}'^{\bm{*}}\in \phi^{NE}(G'^G)$, where $\bm{o}'^{\bm{*}}$ is the constant function with value $o'^*$.
    \end{enumerate}
\end{theorem}
\begin{proof}
    \begin{enumerate}
        \item Assume that $o^*\in\phi^{NE}(G)$,   $o'^*\in\phi^{NE}(G')$ and that for some $i\in I$, $\pair{o^*,o'^*}R^{\times_I}_i\pair{p,p'}$. Then, by the definition of   $R^{\times_I}_i$, $o^*R_ip$ and $o'^*R'_ip'$. Since $o^*$ and $o'^*$ are equilibria, it turns out that $p\preceq_i o^*$ and $p'\preceq_i o'^*$, from where $\pair{p,p'}\preceq_i^{\times_I}\pair{o^*,o'^*}$. 
        
        To prove the converse implication, consider $\pair{o^*,o'^*}\in\phi^{NE}(G\times_I G')$ and that for some $i\in I$, $o^*R_ip$. Then, since $R'_i$ is reflexive,  $o'^*R'_io'^*$ so $\pair{o^*,o'^*}R^{\times_I}_i\pair{p,o'^*} $, therefore   $p\preceq_i o^*$. The same argument proves that $o'^*\in\phi(G')$.
          \item Assume that $o^*\in\phi^{NE}(G)$.  If we  assume that for some $i\in I$,  $\pair{i_O({o^*}),x}\in R_i+R'_i$, then we must have  $x=i_O(p)$ for some $p\in O$ and  ${o^*}R_ip$. By the hypothesis, $p\preceq_i o^*$ and thus $o^*\in \phi^{NE}(G+_IG')$. 
  
  If now  we assume that $i_O(o^*)\in\phi^{NE}(G+_IG')$ and for some $i\in I$ and $p\in O$, $o^*R_ip$, then $\pair{i_O(o^*),i_O(p)}\in R_i+R'_i$, so by hypothesis   $i_O(p)(\preceq_i+\preceq'_i)i_O(o^*)$. It follows that $p\preceq_io^*$, so $o^*\in\phi^{NE}(G)$.
  
    \item Assume that $o'^*\in\phi^{NE}(G')$. Consider for a fixed $i\in I$,  $f\in \mathcal{O}$ such that $\pair{f,\bm{o}'^{\bm{*}}}\in R''_i$. This means that for every $o,p\in O$, if $oR_ip$ then $f(o)R'_i\bm{o}'^{\bm{*}}(p)$. In particular, since $oR_io$ we have by hypothesis that $f(o)\preceq'_io'^*$. Thus $f\preceq''_i\bm{o}'^{\bm{*}}$.

For the converse, consider $o'^*\in O'$ such that $\bm{o}'^{\bm{*}}\in \phi^{NE}(G'^G)$. Then for any $i\in I, f\in \mathcal{O}$, if $fR''_i\bm{o}'^{\bm{*}}$ then $f\preceq''_i\bm{o}'^{\bm{*}}$. We want to show that $o'^*\in \phi^{NE}(G')$. For this suppose that  $p'\in O'$ is such that $p'R'_io'^*$. Therefore, for the constant function $\bm{p'}$, $\bm{p'}R''_i\bm{o}'^{\bm{*}}$ also holds, so by hypothesis, $\bm{p'}\preceq''_i\bm{o}'^{\bm{*}}$. This means that for all $o\in O$, $\bm{p'}(o)\preceq'_i\bm{o}'^{\bm{*}}(o)$, i.e. $p'\preceq'_io'^*$.
  
    \end{enumerate}
\end{proof}

A similar result holds in \gam, with a correspondingly different proof: 

\begin{theorem}~\label{propNE}
    Given games $G$ and $G'$ in \gam\ with non-empty sets of players, then
    \begin{enumerate}
        \item $o^*\in\phi^{NE}(G)$ and $o'^*\in\phi^{NE}(G')$ if and only if $\pair{o^*,o'^*}\in\phi^{NE}(G\times G')$.
        \item  $o^*\in\phi^{NE}(G)$   if and only if $i_G(o^*)\in \phi^{NE}(G+ G')$. Similarly for equilibria in  $\phi^{NE}(G')$. 
        \item $o'^*\in\phi^{NE}(G')$ if and only if $\bm{o}'^{\bm{*}}\in \phi^{NE}(G'^G)$, where $\bm{o}'^{\bm{*}}$ is the constant function with value $o'^*$.
    \end{enumerate}
\end{theorem}
\begin{proof}
    \begin{enumerate}
\item Assume that $o^*\in\phi^{NE}(G)$,   $o'^*\in\phi^{NE}(G')$. Thus for every  $i\in I$, $p\in O$, if  $o^*R_ip$, then $p\preceq_io^*$, and for all $j\in J$, $p'\in O'$, if  $o'^*R'_jp'$, then $p'\preceq'_jo'^*$. Now consider any pair $\pair{i,j}\in I\times J$ such that $\pair{o^*,o'^*}R_{\pair{i,j}}\pair{p,p'}$. Then, by the definition of   $R_{\pair{i,j}}$, $o^*R_ip$ 
 and $o'^*R'_jp'$. Since $o^*$ and $o'^*$ are equilibria, it turns out that $p\preceq_i o^*$ and  $p'\preceq'_j o'^*$, from where $\pair{p,p'}  \preceq_{\pair{i,j}}\pair{o^*,o'^*}$.
        
  To prove the converse implication, consider $\pair{o^*,o'^*}\in\phi^{NE}(G\times G')$ and that for some $i\in I$, $o^*R_ip$. Then, since for any $j\in J$, $R'_j$ is reflexive, $o'^*R'_jo'^*$ so  $\pair{o^*,o'^*}R_{\pair{i,j}}\pair{p,o'^*}$, so $\pair{p,o'^*}\preceq_{\pair{i,j}}\pair{o^*,o'^*}$ therefore   $p\preceq_i o^*$. The same argument proves that $o'^*\in\phi^{NE}(G')$.

  \item Assume that $o^*\in\phi^{NE}(G)$.  If we  assume that for some $k\in I+J$,  $i_O({o^*})R^+_kx$, then we must have $k\in I$, and $x=i_O(p)$ for some $p\in O$. Then  ${o^*}R_kp$. By the hypothesis, $p\preceq_k o^*$, so  $x\preceq^+_ki_O(o^*)$  and thus $o^*\in \phi^{NE}(G+G')$. 
  
  If now  we assume that $i_O(o^*)\in\phi^{NE}(G+G')$ and for some $i\in I$ and $p\in O$, $o^*R_ip$, then $i_O(o^*)R^+_ii_O(p)$, so by hypothesis
  $i_O(p)\preceq^+_ii_O(o^*)$. It follows that $p\preceq_io^*$, so $o^*\in\phi^{NE}(G)$.

  \item Assume that $\pair{f,\bm{o}'^{\bm{*}}}\in R_\kappa$, for some $\kappa:I\to J$. Then $(\kappa,f)$ and $(\kappa,\bm{o}'^{\bm{*}})$ are morphisms from $G$ to $G'$ in \gam, and for all $i\in I, o,p\in O$, if $oR_ip$, then $f(o)R'_{\kappa(i)}\bm{o}'^{\bm{*}}(p)$. By the hypothesis that $o'^*\in \phi^{NE}(G')$, it turns out that $f(o)\preceq'_{\kappa(i)}\bm{o}'^{\bm{*}}(p)$, so $f\preceq_\kappa\bm{o}'^{\bm{*}}$.

  If $\bm{o}'^{\bm{*}}$ is a Nash equilibrium in $G'^G$, and $p'R'_jo'^* $ for some $j\in J$, then consider the constant functions $\bm{p'}$ and $\bm{j}$. By an argument similar to that in Proposition \ref{embeddingGam}, we know that $(\bm{j},\bm{p}')$ and $(\bm{j},\bm{o}'^{\bm{*}})$ are morphisms.  Furthermore, $\bm{p}'R_{\bm{j}}\bm{o}'^{\bm{*}} $, so by hypothesis, $\bm{p}'\preceq_{\bm{j}}\bm{o}'^{\bm{*}}$. Therefore, for all $i\in I$ and $o\in O$, $\bm{p}'(o)\preceq'_{\bm{j}(i)}\bm{o}'^{\bm{*}(o)}$, proving that $p'\preceq'_jo'^*$ as required. \qedhere
    \end{enumerate}
\end{proof}

A similar result establishing the preservation of weak Nash equilibria under products, coproducts, and exponentials can be proved. However, since its proof follows analogous arguments and introduces no fundamentally new ideas, we omit it for the sake of brevity.

Furthermore, we can also find the Nash equilibria in coproducts of games in which we identify some subset of players:

\begin{proposition}\label{propPushout}
    Given games $G = \langle I,O,\{R_{i}\}_{i \in I}, \{\preceq _{i}\}_{i \in I}\rangle$ and $G' = \langle J, O', \{R'_j\}_{j \in J}, \{\preceq '_j\}_{j \in J} \rangle$, such that  $I\cap J=S\neq\emptyset$, let $G+G'/_{\sim_S}$ be the pushout  based on $G_p(S)$ as in Example \ref{examplePushout}. Then $o^*\in\phi^{NE}(G)$ if and only if $i_G(o^*)\in\phi^{NE}(G+G'/_{\sim_S})$.
\end{proposition}

\begin{proof} Because of the way we defined the pushout, the equivalence classes of outcomes are singletons and therefore we just identify them with their single element.  Assume that $o^*\in\phi^{NE}(G)$ and $i_G(o^*)R^+_{[i]}p$ for some $i\in I+J$. This means that there exists some $j\in[i]$ such that $i_G(o^*)R^+_{j}p$. It follows that $p$ must be the inclusion of some outcome $p'\in O$ and thus $i_G(o^*)R^+_{j}i_G(p')$, so $o^*R_jp'$. Then $p'\preceq_j o^*$ and therefore $p\preceq^+_j i_G(o^*)$. In other words $p\preceq^+_{[i]} i_G(o^*)$.

For the converse, assume that $o^*R_kp$ for some $k\in I$. Then  $i_G(o^*)R^+_{[k]}i_G(p)$. Since by hypothesis $i_G(o^*)\in\phi^{NE}(G+G'/_{\sim_S})$, we have two cases: 
\begin{itemize}
\item if $k\in I\setminus S$, $[k]=\set{\pair{k,0}}$ and thus $i_G(o^*)R^+_{\pair{k,0}}i_G(p)$ so $i_G(p)\preceq^+_{\pair{k,0}}i_G(o^*)$ and therefore $p\preceq_{k}o^*$.
\item if $k\in S$, $[k]=\set{\pair{k,0},\pair{k,1}}$ but  since $i_G(o^*), i_G(p)\in i_G(O)$ this means that $i_G(o^*)R^+_{\pair{k,0}}i_G(p)$ so $p\preceq_{k}o^*$ as before.
\end{itemize}
\end{proof}

\section{Categories with equilibria-preserving morphisms} \label{sectNash}

We can regard $\phi^{NE}$ as an operator taking games to games, but in general, it is not a functor in \gam. To see this, we observe that even though game morphisms preserve the accessibility relations and preferences of the players, this is not enough to make sure that Nash equilibria will be preserved:

\begin{example}~\label{exi}
Consider two simple two-player games:  $G$ with $O=\set{a,b}$ and $R_1=R_2=\preceq_1=\preceq_2=1_O$ and $G'$ with $O'=O$, $R'_1=R'_2=Id_{O}\cup\set{\pair{a,b}, \pair{b,a}}$, $\preceq'_1=Id_{O}\cup\set{\pair{a,b}}$, and $ \preceq'_2=Id_{O}\cup\set{\pair{b,a}}$. The identity function on $O$ gives a morphism. Both outcomes are Nash equilibria in $G$ but there is no equilibrium in $G'$. 
\end{example}

We see in the example above that the identity function $Id_O$ gives a morphism from $G\to G'$, but since there are no outcomes in $\phi^{NE}(G')$, there is no way to pick a morphism from $\phi^{NE}(G)$ to $\phi^{NE}(G')$ to fulfill the role of $\phi^{NE}(Id_O)$.

 Notice that this negative result applies to both \gam\ and $\mathbf{Gam}_{I}$. A way of addressing this issue is by taking hints from \cite{Lapitsky1999} and \cite{ghani2018compositional} in which the idea of preservation of equilibria is built in the definition of morphism. In this spirit, we look at some subcategories of  $\mathbf{Gam}$ with equilibria-preserving morphisms.
 
 \begin{definition} Let \textbf{NE} be the category of games in which all the outcomes are Nash equilibria. This is a full subcategory of \gam, meaning that all the morphisms in \gam\ between two objects in \textbf{NE} are also morphisms in \textbf{NE}.
 \end{definition}

 Notice that if $G$ is a game in \textbf{NE}, $\phi^{NE}({G}) = {G}$. That is, objects in \textbf{NE} are fixed points under $\phi^{NE}$.

\begin{definition}
    We say that a morphism $f:G\to G'$ in \gam\  \emph{preserves Nash equilibria} if $f_O(\phi^{NE}(G))\subseteq \phi^{NE}(G')$.
\end{definition}

All the morphisms in \textbf{NE} trivially preserve Nash equilibria. Motivated by this, we can also define another  subcategory of \gam\  with equilibria preserving morphisms:

 \begin{definition} Let \gam$^{NE}$ be the subcategory of \gam\ with all the same objects, but only the equilibria-preserving morphisms.
 \end{definition}

It can be easily checked that \gam$^{NE}$\  is a category since identities and compositions of equilibria-preserving morphisms are equilibria-preserving. Furthermore, \textbf{NE} is a subcategory of \gam$^{NE}$ since as pointed out before, all its morphisms trivially preserve equilibria.

As a consequence of Theorem \ref{propNE} and Proposition \ref{propPushout} we have:

\begin{corollary}
    Products, coproducts and pushouts based on player games can be constructed in \gam$^{NE}$.
\end{corollary}

A sufficient requirement ensuring that a morphism preserves Nash equilibrium is as follows:

\begin{lemma}~\label{Nash}
Let $f:G\to G'$ in $\mathbf{Gam}$ be such that $f_p: I \rightarrow J$ and  $f_O: O\to O'$ are {\em surjective} and for each $i \in I$ and every pair $o, p \in O$:
			\[  oR_ip \text{ iff } f_O(o)R'_{f_p(i)} f_O(p).\]
  Then, if $o^*$ is a Nash equilibrium in $G$,  $f_O(o^*)$ is a Nash equilibrium in $G'$. That is, $f$ is a morphism in \gam$^{NE}$.
\end{lemma}

 \begin{proof}  Assume that $f_O(o^*)$ is not a Nash equilibrium. Then, there exist $j\in J$ and $o'\in O'$ such that \[ f_O(o^*)R'_{j} o',\ \text{ while } \ o'\not\preceq'_{j} f_O(o^*).\]
Since $f_p$ and $f_O$ are surjective, there exists $i \in I$ and $o\in O$ such that $f_p(i) = j$ and $f_O(o)=o'$, respectively. Then, we have that $f_O(o^*)R'_{f_p(i)} f_O(o)$  and $f_O(o)\not\preceq'_{f_p(i)} f_O(o^*)$. By the assumptions on $f$, we get that $o^*R_io$ and since $o^*$ is a Nash equilibrium, $o\preceq_i o^*$, so $f_O(o)\preceq'_{f_p(i)} f_O(o^*)$, which is a contradiction.
\end{proof}

\noindent {\bf Example}: {Consider the games of examples~\ref{PD} and \ref{BoS}, taking the players of game ${PD}$ to be $I_{PD} =\{1,2\}$ while those of $G_{BoS}$ are, $I_{BoS}=\{2,3\}$.  That is, $2$ plays both games. The pushout game as defined in Example \ref{examplePushout}, can be denoted $G_{PD} +_2 G_{BoS}$.}

By the result in Theorem \ref{propNE}, 2, the Nash equilibria of $G_{PD}$ and  $G_{BoS}$, namely $(D,D), r$ and $s$ are included in the coproduct. Since in the construction of the pushout the equivalence classes of the outcomes are all singletons, those are the equilibria of  $G_{PD} +_2 G_{BoS}$.

	\bibliographystyle{alpha}

\end{document}